\documentclass{amsart}

\usepackage{amsfonts,amsmath,mathrsfs,amsthm,amssymb}
\usepackage{lmodern}
\usepackage[T1]{fontenc}
\usepackage{hyperref}
\hypersetup{colorlinks,urlcolor=blue,linkcolor=blue,citecolor=red}
\usepackage{tikz}
\usetikzlibrary{shapes,arrows,decorations.markings,positioning}
\usepackage[ruled,noline,linesnumbered]{algorithm2e}
\usepackage{xspace}
\usepackage{enumitem}
\usepackage{leftidx}

\newtheorem{theorem}{Theorem}[section]
\newtheorem{lemma}[theorem]{Lemma}
\newtheorem{proposition}[theorem]{Proposition}
\newtheorem{corollary}[theorem]{Corollary}

\theoremstyle{definition}
\newtheorem{definition}[theorem]{Definition}
\newtheorem{example}[theorem]{Example}
\newtheorem{assumption}{Assumption}

\theoremstyle{remark}
\newtheorem{remark}[theorem]{Remark}

\numberwithin{equation}{section}

\setlist{leftmargin=2.5em}
\setlist[itemize]{label={--},leftmargin=1.5em}
\renewcommand{\geq}{\geqslant}
\renewcommand{\leq}{\leqslant}

\newcommand{\R}{\mathbb{R}}
\newcommand{\Z}{\mathbb{Z}}
\newcommand{\N}{\mathbb{N}}
\newcommand{\C}{\mathcal{C}}
\newcommand{\E}{\mathcal{E}}

\newcommand{\M}{\mathcal{M}}

\newcommand{\unit}{e}
\newcommand{\state}{S}
\newcommand{\MIN}{{\scshape Min}\xspace}
\newcommand{\MAX}{{\scshape Max}\xspace}
\newcommand{\polMIN}{\mathcal{S}_{\mathrm p}}
\newcommand{\polMAX}{\mathcal{T}_{\mathrm p}}
\newcommand{\chibar}{\overline{\chi}}
\newcommand{\transpose}[1]{{#1}^\intercal}
\newcommand{\Hnorm}[1]{\| \ifx\\#1\\ \cdot \else #1 \fi \|_\text{H}}

\def\<#1,#2>{#1 \cdot #2}
\newcommand{\ki}{k}

\begin{document}

\title[Generic uniqueness of the bias of finite stochastic games]{Generic uniqueness of the bias vector of finite zero-sum stochastic games with perfect information}

\author[M.~Akian]{Marianne Akian}
\author[S.~Gaubert]{St\'ephane Gaubert}
\author[A.~Hochart]{Antoine Hochart}
\address{INRIA Saclay-Ile-de-France and CMAP, Ecole polytechnique, Route de Saclay, 91128 Palaiseau Cedex, France}
\email{marianne.akian@inria.fr}
\email{stephane.gaubert@inria.fr}
\email{antoine.hochart@polytechnique.edu}
\thanks{A.~Hochart has been supported by a PhD fellowship of Fondation Math\'ematique Jacques Hadamard (FMJH). The authors are also partially supported by the PGMO program of EDF and FMJH and by the ANR (MALTHY Project, number ANR-13-INSE-0003).}

% 47J10: Nonlinear spectral theory, nonlinear eigenvalue problems


\date{May 24, 2017}

\begin{abstract}
  Mean-payoff zero-sum stochastic games can be studied by means of a nonlinear spectral problem.
  When the state space is finite, the latter consists in finding an eigenpair $(u,\lambda)$
  solution of $T(u)=\lambda \unit + u$, where $T:\R^n \to \R^n$ is the Shapley
  (or dynamic programming) operator, $\lambda$ is a scalar, $\unit$ is the unit vector,
  and $u \in \R^n$.
  The scalar $\lambda$ yields the mean payoff per time unit, and the vector $u$,
  called {\em bias}, allows one to determine optimal stationary strategies
  in the mean-payoff game.
  The existence of the eigenpair $(u,\lambda)$ is generally related to ergodicity conditions.
  A basic issue is to understand for which classes of games the bias vector is unique
  (up to an additive constant).
  In this paper, we consider perfect-information zero-sum stochastic games with finite state
  and action spaces, thinking of the transition payments as variable parameters,
  transition probabilities being fixed.
  We show that the bias vector, thought of as a function of the transition payments,
  is generically unique (up to an additive constant).
  The proof uses techniques of 
  nonlinear Perron-Frobenius theory.
  As an application of our results, we obtain an explicit perturbation scheme
  allowing one to solve degenerate instances of stochastic games by policy iteration.
\end{abstract}

\maketitle

\section{Introduction}

\subsection{The ergodic equation for stochastic games}

Repeated zero-sum games describe long-term interactions between two agents, called players, with opposite interests.
In this paper, we consider {\em perfect-information zero-sum stochastic games}, in which the players choose repeatedly and alternatively an action, being informed of all the events that have previously occurred (state of nature and chosen actions).
These choices determine at each stage of the game a payment, as well as the next state by a stochastic process.
Given a finite horizon $k$ and an initial state $i$, one player intends to minimize the sum of the payments of the $k$ first stages, while the other player intends to maximize it.
This gives rise to the value of the $k$-stage game, denoted by $v_i^k$.

A major topic in the theory of zero-sum stochastic games is the asymptotic behavior of the mean values per time unit $(v_i^k/k)$ as the horizon $k$ tends to infinity.
The limit, when it exists, is referred to as the {\em mean payoff}.
This question was first addressed in the case of a finite state space by Everett~\cite{Eve57}, Kohlberg~\cite{Koh74}, and Bewley and Kohlberg~\cite{BK76}.
See also Rosenberg and Sorin~\cite{RS01}, Sorin~\cite{Sor04}, Renault~\cite{Ren11}, Bolte, Gaubert and Vigeral~\cite{BGV15} and Ziliotto~\cite{ziliotto}
for more recent developments.
We refer the reader to~\cite{NS03} for more background on stochastic games.

A way to study the asymptotic behavior of the values $(v_i^k)_k$ consists in exploiting the recursive structure of the game.
This structure is encompassed in the {\em dynamic programming} or {\em Shapley operator} of the game.
In this paper, since the state space is assumed to be finite, say $\{1,\dots,n\}$, the latter is a map $T: \R^n \to \R^n$.
Then, a basic tool to study the asymptotic properties of the sequence $(v_i^k)_k$ is the following nonlinear spectral problem, called the {\em ergodic equation}:
\begin{equation}
  \label{eq:ergodic-equation}
  T(u) = \lambda \unit + u \enspace ,
\end{equation}
where $\unit$ is the unit vector of $\R^n$.
Indeed, if there exist a vector $u \in \R^n$ and a scalar $\lambda \in \R$ solution of~\eqref{eq:ergodic-equation}, then, not only the sequence $(v_i^k/k)_k$ converges as the horizon $k$ tends to infinity, but also the limit is independent of the initial state $i$, and equal to $\lambda$.
This scalar, which is unique, is called the {\em ergodic constant} or the {\em (additive) eigenvalue} of $T$, and the vector $u$, called {\em bias vector} or {\em (additive) eigenvector}, gives optimal stationary strategies in the mean-payoff game. 

A first question is to understand when the ergodic equation is solvable.
In \cite{AGH15}, we considered the Shapley operator $T: \R^n \to \R^n$ of a game with finite state space and bounded transition payment function, 
and gave necessary and sufficient conditions under which the ergodic equation is solvable for all the operators $g+T$ with $g \in \R^n$, or equivalently for all perturbations $g$ of the payments that only depend on the state.
Moreover, assuming the compactness of the action sets of players and some continuity of the transition functions, these conditions can be characterized in terms of reachability in directed hypergraphs.

A second
question concerns the structure of the set of bias vectors. 
For one-player problems, i.e., for discrete optimal control, the ergodic equation~\eqref{eq:ergodic-equation}, also known as the {\em average case optimality equation}, has been much studied, either in the deterministic or in the stochastic case (Markov decision problems).
Then, the representation of bias vectors and their relation with optimal strategies is well understood.

Indeed, in the deterministic case, the analysis of the ergodic equation relies on max-plus spectral theory, which goes back to the work of Romanovsky~\cite{Rom67}, Gondran and Minoux~\cite{GM77} and Cuninghame-Green~\cite{CG79}.
Kontorer and Ya\-ko\-venko~\cite{KY92}
deal specially with infinite horizon optimization and mean-payoff problems.
We refer the reader to the monographies~\cite{BCOQ92, KM97,butkovic}
or surveys~\cite{Bap98,ABG13} for more background on max-plus spectral theory.
One of the main result of this theory shows that the set of bias vectors has the structure of a max-plus (tropical) cone, i.e., that it is invariant by max-plus linear combinations, and it has a unique minimal generating family consisting of certain ``extreme'' generators, which can be identified by looking at the support of the maximizing measures in the linear programming formulation of the optimal control problem, or at the ``recurrence points'' of infinite optimal trajectories.
A geometric approach to some of these results, in terms of polyhedral fans, has been recently given by Sturmfels and Tran~\cite{ST13}.

The eigenproblem~\eqref{eq:ergodic-equation} has been studied in the more general infinite-dimensional state space case, see
Kolokoltsov and Maslov~\cite{KM97},
Mallet-Paret and Nussbaum~\cite{Nuss-Mallet}, and
 Akian, Gaubert and Walsh~\cite{AGW09} for an approach in terms of horoboundaries. It has also been studied in the setting of weak KAM theory, for which we refer the reader to Fathi~\cite{Fat}.
In the weak KAM setting, the bias vector becomes the solution
of an ergodic Hamilton-Jacobi PDE; Figalli and Rifford showed
in~\cite{rifford} that this solution is unique for generic
perturbations of a Hamiltonian by a potential function.
In the stochastic case, the structure of the set of bias vectors is still known when the state space is finite, see Akian and Gaubert~\cite{AG03}. 

In the two-player case, the structure of the set of bias vectors is less well known, although the description of this set remains a fundamental issue.
In particular, the uniqueness of the bias vector up to an additive constant
is an important matter for algorithmic purposes.
Indeed, the nonuniqueness of the bias typically leads to numerical instabilities or degeneracies.
In particular, the standard Hoffman and Karp policy iteration algorithm~\cite{HK66} may fail to converge in situations in which the bias vector is not unique.
The standard approach to handle such degeneracies is to approximate
the ergodic problem by the discounted problem. This
was proposed by Puri~\cite{puri}
in the case of deterministic games and this was further analyzed
by Paterson and Zwick~\cite{zwick}, see also the discussion by Friedmann in~\cite{friedmann}.
A different approach was proposed 
by Cochet-Terrasson and Gaubert~\cite{CTG06}, Akian, Cochet-Terrasson, Detournay and Gaubert~\cite{ACTDG}, and by Bourque and Raghavan~\cite{BR14}, allowing one to circumvent such degeneracies at the price of an increased complexity of the algorithm (handling the nonuniqueness of the bias).
Hence, it is of interest to understand when such technicalities can be avoided.

\subsection{Main results}

We address the question of the uniqueness of the bias vector for stochastic games with perfect information and finite state and action spaces, restricting our attention to games for which the ergodic equation~\eqref{eq:ergodic-equation} is solvable {\em for all state-dependent perturbations of the transition payments}.
Our main result, Theorem~\ref{thm:generic-uniqueness}, shows that the bias vector is generically unique up to an additive constant.
More precisely, we show that the set of perturbation vectors for which the bias vector is not unique belongs to a polyhedral complex the cells of which have codimension one at least.
A first ingredient in the proof relies on nonlinear Perron-Frobenius theory~\cite{AG03}.
A second ingredient is a general result, showing that the set of fixed points of a nonexpansive self-map of $\R^n$ is a retract of $\R^n$, see Theorem~\ref{thm:nonexpansive-retract}.
This allows us to infer the uniqueness of the bias vector of a Shapley operator from the uniqueness of the bias vector of the reduced Shapley operators obtained by fixing the strategy of one player.

We then present an algorithmic application of our results.
Hoffman and Karp introduced a policy iteration algorithm to solve mean-payoff zero-sum
stochastic games with perfect information and finite state and action spaces~\cite{HK66}.
They showed that policy iteration does terminate if every pair of strategies
of the two players yields an irreducible Markov chain.
If this irreducibility assumption is not satisfied, policy iteration may cycle.
However, the irreducibility assumption is not satisfied by many classes of games
-- in particular, it is essentially never satisfied for deterministic games.
The cycling of the Hoffman-Karp algorithm is due to the nonuniqueness of the bias vector. 
Hence, we deduce from our results that the Hoffman-Karp policy iteration
algorithm does converge if the payment is generic.
Moreover, we provide a family of effective perturbations of the payment for which
the bias vector is unique (up to an additive constant).
This leads to an explicit perturbation scheme, allowing one to solve nongeneric instances
by policy iteration, avoiding the classical irreducibility condition
or the use of vanishing discount based perturbation schemes.

The paper is organized as follows.
After some preliminaries on stochastic games in Section~\ref{sec:preliminaries}, we establish
the generic uniqueness of the bias vector in Section~\ref{sec:generic-uniqueness}.
The application to policy iteration is presented in Section~\ref{sec:applications}.

We finally point out that some of the present results have been announced 
in the conference article~\cite{cdc2014}.

\section{Preliminaries on zero-sum stochastic games}
\label{sec:preliminaries}

\subsection{Games with perfect information}
\label{sec:prel1}
In this paper we consider {\em finite (zero-sum) stochastic games with perfect information}, where the second player, called player \MAX, always makes a move after being informed of the action chosen by the first player, called player \MIN.
Such a game is characterized by the following:
\begin{itemize}
  \item a finite state space, that we denote by $\state = \{1,\dots,n\}$;
  \item finite action spaces, denoted by $A_i$ for player \MIN when the current state is $i \in \state$, and by $B_{i,a}$ for player \MAX in state $i$ and after action $a \in A_i$ has been chosen by player \MIN;
  \item a transition payment $r_i^{a b} \in \R$, paid by player \MIN to player \MAX, when the current state is $i \in \state$ and the last actions selected by the players are $a \in A_i$ and $b \in B_{i,a}$;
  \item a transition probability $P_i^{a b} \in \Delta(\state)$ (where $\Delta(\state)$ denotes the set of probability measures over $\state$) which gives the law according to which the next state is determined, when the current state is $i \in \state$ and the last actions selected by the players are $a \in A_i$ and $b \in B_{i,a}$.
\end{itemize}
Denoting by $K_A := \bigcup_{i \in \state} \{i\} \times A_i$ the action
set of player \MIN
and by $K_B := \bigcup_{(i,a) \in K_A} \{(i,a)\} \times B_{i,a}$ the action
set of
player \MAX, a finite stochastic game with perfect information is thus defined
by a 5-tuple $\Gamma := (\state, K_A, K_B, r, P)$ where the three first sets are finite.

Such a game is played in stages, starting from a given initial state $i_0$, as follows: at step $\ell$, if the current state is $i_\ell$, player \MIN chooses an action $a_\ell \in A_{i_\ell}$, and player \MAX subsequently chooses an action $b_\ell \in B_{i_\ell, a_\ell}$.
Then, player \MIN pays $r_{i_\ell}^{a_\ell b_\ell}$ to player \MAX and the next state is chosen according to the probability law $P_{i_\ell}^{a_\ell b_\ell}$.
The information is perfect, meaning that at each step, the players have a perfect knowledge of all the previously chosen actions, as well as the state previously visited.

Denote by $H_k^\text{\MIN} := (K_B)^k \times \state$ the set of histories of length $k$
of player \MIN and by $H^\text{\MIN} := \bigcup_{k \in \N} H_k^\text{\MIN}$ the set of
all finite histories of player \MIN.
Likewise, define $H_k^\text{\MAX} := (K_B)^k \times K_A$ and
$H^\text{\MAX} := \bigcup_{k \in \N} H_k^\text{\MAX}$ to be, respectively, the set of
histories of length $k$ and the set of all finite histories of player \MAX.
Let $H_\infty := (K_B)^\N$ be the set of infinite histories.
A strategy of player \MIN is a map
\[
  \sigma: H^\text{\MIN} \to \bigcup_{i \in \state} \Delta(A_i)
\]
(where $\Delta(X)$ denote the set of probability measures over any set $X$) such that,
for every finite history $h_k = (i_0,a_0,b_0,\dots,i_k) \in H_k^\text{\MIN}$,
we have $\sigma(\cdot \mid h_k) \in \Delta(A_{i_k})$.
Likewise, a strategy of player \MAX is a map
\[
  \tau: H^\text{\MAX} \to \bigcup_{(i,a) \in K_A} \Delta(B_{i,a})
\]
such that, for every finite history
$h'_k = (i_0,a_0,b_0,\dots,i_k,a_k) \in H_k^\text{\MAX}$, we have
$\tau(\cdot \mid h'_k) \in \Delta(B_{i_k,a_k})$.

A strategy $\sigma$ (resp.\ $\tau$) of player \MIN (resp.\ \MAX) is pure and Markovian if it depends only on the current stage and state 
(resp.\ the current stage, state and action of player \MIN) and is deterministic,
that is its values are Dirac probabilities.
Such a strategy is stationary if it does not depend on stage.
We denote by $\polMIN$ the finite set of (deterministic) 
{\em policies} of player \MIN, i.e., the set of maps
$\sigma: \state \to \bigcup_{i \in \state} A_i$ such that $\sigma(i) \in A_i$
for every state $i \in \state$.
Likewise, we denote by $\polMAX$ the set of (deterministic) 
policies of player \MAX, i.e., the set of maps
$\tau: K_A \to \bigcup_{i \in \state, a \in A_i} B_{i,a}$
such that $\tau(i,a) \in B_{i,a}$ for every state $i \in \state$ and action $a \in A_i$.
Then, a pure Markovian strategy of player \MIN (resp.\ \MAX) 
can be identified to a sequence of policies:
for  every history of length $k$ of player \MIN, $h_k = (i_0,a_0,b_0,\dots,i_k) \in H_k^\text{\MIN}$,
$\sigma(\cdot\mid h_k)$ is the Dirac measure at $\sigma_k(i_k)$.
Likewise, for every history of length $k$ of player \MAX,
$h'_k = (i_0,a_0,b_0,\dots,i_k,a_k) \in H_k^\text{\MAX}$, 
$\tau(\cdot \mid h'_k)$ is the Dirac measure at $\tau_k(i_k,a_k)$.
Moreover, a pure Markovian stationary strategy can be identified to a single policy.

Given an initial state $i$ and strategies $\sigma$ and $\tau$ of the players, the triple
$(i,\sigma,\tau)$ defines a probability law on $H_\infty$ the expectation of which is denoted
by $\mathbb{E}_{i,\sigma,\tau}$.
The {\em payoff} of the $k$-stage game is the following additive function
of the transition payments:
\begin{equation*}
  \label{eq:finite-horizon-payoff}
  J_i^k(\sigma,\tau) := \mathbb{E}_{i,\sigma,\tau}
  \Bigg[ \sum_{\ell=0}^{k-1} r_{i_\ell}^{a_\ell b_\ell} \Bigg] \enspace .
\end{equation*}
Player \MIN intends to choose a strategy minimizing the payoff $J_i^k$, whereas player \MAX
intends to maximize the same payoff.
The value of the $k$-stage game starting at state $i$ is then defined by
\[
  v_i^k := \inf_\sigma \sup_\tau J_i^k(\sigma,\tau) =
  \sup_\tau \inf_\sigma J_i^k(\sigma,\tau) \enspace ,
\]
when the equality holds, 
where the infimum and the supremum are taken over the set of all strategies of
players \MIN and \MAX, respectively.
Since the game has finite state and action spaces, the value exists,
with the infima and suprema realized by pure Markovian strategies
of both players~\cite{Sha53}.

In this paper, we are interested in the asymptotic behavior of the sequence of mean values per time unit $(v^k/k)_{k \geq 1}$.
When the latter ratio converges, the limit will be called the {\em mean-payoff vector}.
A stronger condition is the existence of a 
{\em uniform value} $v^{\text{U}}$, meaning that
\begin{align}
\inf_{\sigma}\, \limsup_{k \to \infty} \, \sup_{\tau}\; \frac{1}{k} J_i^k(\sigma,\tau) 
\leq v_i^\text{U} 
\leq \sup_{\tau}\,\liminf_{k \to \infty} \,\inf_{\sigma}\;\frac{1}{k} J_i^k(\sigma,\tau) \enspace ,\label{e-def-unifvalue}
\end{align}
where infima and suprema are taken over all strategies
(see for instance~\cite{sorinbook,renault12}).
This implies that $\lim_k v^k/k = v^{\text{U}}$.
Note that the first (resp.\ second) inequality means that player \MIN 
(resp.\ \MAX) uniformly guarantees $v_i^\text{U}$.

Moreover, following~\cite{laraki-renault},
{\em optimal uniform strategies} are defined as strategies $\sigma^*$ and
$\tau^*$ for players \MIN and \MAX respectively such that
\begin{align}
 \limsup_{k \to \infty} \, \sup_{\tau}\; \frac{1}{k} J_i^k(\sigma^*,\tau) =
v_i^\text{U}  =\liminf_{k \to \infty} \,\inf_{\sigma}\;\frac{1}{k} J_i^k(\sigma,\tau^*) \enspace .
\label{e-def-optuniform}
\end{align}
Note that Condition~\eqref{e-def-optuniform} is 
stronger than~\eqref{e-def-unifvalue}.
When $v^\text{U}$ exists, it also coincides with the value of the zero-sum game
in which one considers one the following  {\em limiting average payoffs}:
\begin{equation}
  \label{eq:average-payoffs}
  \begin{aligned}
    J_i^\text{+}(\sigma,\tau) & := \limsup_{k \to \infty} \; \frac{1}{k} J_i^k(\sigma,\tau)  \enspace , \\
    J_i^\text{-}(\sigma,\tau) & := \liminf_{k \to \infty} \; \frac{1}{k} J_i^k(\sigma,\tau) \enspace .
  \end{aligned}
\end{equation}
Moreover, optimal uniform strategies are also optimal for these games.
Mertens and Neyman~\cite{MN81} proved that for stochastic games with finite
state and action spaces and imperfect information the uniform value exists.
It follows that the uniform value also exists for finite state and action
spaces games with perfect information -- the latter can be reduced
to degenerate instances of imperfect-information games in which in each state,
only one of the players has a choice of action.
We shall also recall in Theorem~\ref{th-inv}
how for this class of games, the existence of the uniform value
and of uniform optimal strategies, follows from a result
of Kohlberg.  

In the computer science litterature~\cite{zwick,AM09}, mean-payoff games are defined in a slightly different manner, following Ehrenfeucht and Mycielski~\cite{ehrenfeucht}, as non-zero sum games in which player \MIN\
wishes to minimize $J_i^\text{+}(\sigma,\tau)$ whereas player \MAX\
wishes to maximize $J_i^\text{-}(\sigma,\tau)$. 
Liggett and Lippman~\cite{LL69} showed that such games 
admit optimal policies
$\sigma^*,\tau^*$ and a value $v^*$, meaning that 
\[
J_i^\text{+}(\sigma^*,\tau) \leq 
v^*=J_i^\text{+}(\sigma^*,\tau^*)=
J_i^\text{-}(\sigma^*,\tau^*) \leq J_i^\text{-}(\sigma,\tau^*) 
\]
for all pair of strategies $\sigma,\tau$ of players \MIN\ and \MAX.
The latter property is implied by the existence of the uniform value
and the existence of pure Markovian stationary uniform optimal strategies
(also called uniform optimal policies), so that $v^*=v^{\text{U}}$.

Therefore, in the sequel, we will use the term {\em mean-payoff games} with a general
meaning, understanding that the different approaches that we just discussed
lead to the same notion of value. In particular, the notion of value
and optimal policies can always been understood in the strongest sense
(uniform value and uniform optimal policies).

\subsection{The operator approach}\label{subsec-operator}

The study of the value vector $v^k = (v_i^k)_{i \in \state}$ involves the  {\em dynamic programming operator}, or {\em Shapley operator} of the game.
The latter is a map $T: \R^n \to \R^n$ whose $i$th coordinate is given by
\begin{equation}
  \label{eq:Shapley-operator}
  T_i(x) = \min_{a \in A_i} \max_{b \in B_{i,a}} \left( r_i^{a b} + P_i^{a b} x \right) \enspace , \quad x \in \R^n \enspace.
\end{equation}
Note that an element $P \in \Delta(\state)$ is seen as a row vector $P=(P_j)_{j \in \state}$ of $\R^n$, so that $P x$ means $\sum_{j \in \state} P_j x_j$.
Also note that, given a vector $g \in \R^n$, the operator $g+T$ appears as the Shapley operator
of the game $(\state, K_A, K_B, \tilde{r}, P)$ where the transition payment function satisfies
$\tilde r_i^{a b} = g_i + r_i^{a b}$.
The latter game is almost identical to the initial game $(\state, K_A, K_B, r, P)$,
except that the transition payments are perturbed with quantities that only depend on the state,
hence the designation of $g$ as an {\em additive (state-dependent) perturbation vector}.

The Shapley operator allows one to determine recursively the value vector
of the $k$-stage game:
\begin{equation}
  \label{eq:dynamic-prog-principle}
  v^k = T(v^{k-1}) \enspace , \quad  v^0 = 0 \enspace .
\end{equation}

Also, $T$ is {\em monotone} and {\em additively homogeneous}, meaning that it satisfies the following two properties, respectively:
\begin{align*}
  \tag{monotonicity}
  & x \leq y \implies T(x) \leq T(y) \enspace , \quad x, y \in \R^n \enspace ,\\
  \tag{additive homogeneity}
  & T(x + \lambda \unit) = T(x) + \lambda \unit \enspace , \quad x \in \R^n, \enspace \lambda \in \R \enspace ,
\end{align*}
where $\R^n$ is endowed with its usual partial order
and $\unit$ is the unit vector of $\R^n$.
A first consequence of the additive homogeneity of $T$ is that, for any bias vector $u$ and any scalar $\alpha \in \R$, the vector $u + \alpha \unit$ is also a bias: we say that $u$ is defined up to an {\em additive constant}.
More generally, the importance of the above axioms in stochastic control and game theory has been known for a long time~\cite{CT80}.
In particular, they imply that $T$ is {\em sup-norm nonexpansive}:
\[
  \|T(x) - T(y)\|_\infty \leq \|x-y\|_\infty \enspace , \quad x, y \in \R^n \enspace .
\]

According to the dynamic programming principle~\eqref{eq:dynamic-prog-principle}, the mean-payoff vector defined above is given by
\[
\chi(T) := \lim_{k \to \infty} \frac{T^k(0)}{k} \enspace ,
\]
where $T^k := T \circ \dots \circ T$ denotes the $k$th iterate of $T$.
Observe that, since $T$ is sup-norm nonexpansive, $0$ could be replaced by any vector $x \in \R^n$ in the above limit.

Here, since the action spaces are finite, $T$ is a {\em piecewise affine} map over $\R^n$. Kohlberg showed in~\cite{Koh80} that a piecewise
affine self-map $T$ of $\R^n$ that is nonexpansive in an arbitrary norm
has an {\em invariant half-line},
meaning that there exist two vectors $u, \nu \in \R^n$ such that
\begin{equation}
  \label{eq:invariant-half-line}
  T(u + \alpha \nu) = u + (\alpha + 1) \nu
\end{equation}
for every scalar $\alpha$ large enough.
Kohlberg's theorem applies to the present setting, since the Shapley
operator~\eqref{eq:Shapley-operator} is nonexpansive in the sup-norm.
This implies that the limit $\chi(T)$ does exist and coincides with $\nu$.
It follows that the vector $\nu$ arising in the definition
of invariant half-lines is unique. 

The existence of uniform optimal policies follows from the existence
of an invariant half-line. Moreover, such policies are readily
computed from the invariant half-line.
Although these properties are surely known to some experts, 
we could not find a reference for them, so we next state them as Theorem~\ref{th-inv}.

To this end, we define the {\em reduced Shapley operator} $T^\sigma: \R^n \to \R^n$ associated with the policy $\sigma \in \polMIN$ of player \MIN.
Its $i$th coordinate map is given by
\begin{align}\label{e-def-Tsigma}
  T_i^\sigma(x) = \max_{b \in B_{i,\sigma(i)}} \big( r_i^{\sigma(i) b} + P_i^{\sigma(i) b} x \big), \quad x  \in \R^n \enspace.
\end{align}
Since the action spaces are finite, we readily have, for all $x \in \R^n$,
\begin{equation}
  \label{eq:T-min}
  T(x) = \min_{\sigma \in \polMIN} T^\sigma(x) \enspace ,
\end{equation}
where by $\min$, we mean that for every $x\in\R^n$, the minimum
is attained. Indeed, it suffices to take for $\sigma(i)$ any action
$a\in A_i$ achieving the minimum in~\eqref{eq:Shapley-operator}. 

Similarly, to any policy $\tau \in \polMAX$ of player \MAX,
we associate a {\em dual reduced Shapley operator} $^{\tau} T$,
whose $i$th coordinate map is given  by 
\[
\leftidx{^\tau}{T}{_i}(x) = \min_{a \in A_{i}} \big( r_i^{a \tau(i,a)} + P_i^{a \tau(i,a)} x \big), \quad x  \in \R^n \enspace.
\]
For all $x \in \R^n$, we have
\begin{equation}
  \label{eq:T-max}
  T(x) = \max_{\tau \in \polMAX} \leftidx{^\tau}{T}{}(x) \enspace .
\end{equation}

\begin{theorem}[Coro.\ of~\cite{Koh80}]\label{th-inv}
Perfect information stochastic games with finite state and action spaces
have a uniform value and uniform optimal policies $\sigma^*,\tau^*$
that are obtained as follows:
given an invariant half-line $\alpha \mapsto u+ \alpha \nu$ of the Shapley operator $T$, take 
for $\sigma^*$ any policy $\sigma$ attaining the minimum in~\eqref{eq:T-min}
when $x$ is substituted by $u+\alpha \nu$ for $\alpha$ large enough.
Similarly, take 
for $\tau^*$ any policy $\tau$ attaining the maximum in~\eqref{eq:T-max}
when $x$ is substituted by the same quantity.
\end{theorem}
\begin{proof}
Recall that the {\em germ} of a real function $\R\to \R, \; \alpha \mapsto f(\alpha)$, at the point $+\infty$ is the equivalence class of $f$ modulo the relation
$f\sim g$ if there exists $\alpha_0$ such that $f(\alpha) =g(\alpha)$
for $\alpha\geq \alpha_0$. We shall use the same notation
$f$ for the function and its equivalence class. We shall consider in particular the set $\mathbb{A}$ of germs of affine functions. 
Observe that $\mathbb{A}$ is invariant
by linear combinations, by translation by a constant, and also by the operations $\min$ and $\max$, because $\mathbb{A}$ is totally ordered
(of two affine functions of $\alpha$, one ultimately dominates the other
as $\alpha\to \infty$).
Since the action spaces are finite, if $f\in \mathbb{A}^n$, we may identify $\alpha \to T(f(\alpha))$ to an element of $\mathbb{A}^n$. In other words, $T$
acts on vectors of germs of affine functions. 
We now substitute $x=u+ \alpha \nu$
in~\eqref{eq:Shapley-operator}.
Since $\mathbb{A}$ is totally ordered, the minima
and maxima in~\eqref{eq:Shapley-operator} are attained by actions independent of $\alpha$ provided that $\alpha$ is large enough.
I.e., there exists policies $\sigma^*$ and $\tau^*$ of players \MIN and \MAX
respectively and a constant $\alpha_0$ such that for
all $\alpha\geq \alpha_0$, 
\begin{align}\label{e-invhalf}
u+(\alpha +1) \nu = T(u+ \alpha \nu) = T^{\sigma^*}(u+ \alpha \nu) = \leftidx{^{\tau^*}}{\!T}{}(u+\alpha \nu)  \enspace .
\end{align}
It will now be convenient to use the operator
$\leftidx{^\tau}{T}{^\sigma}$ from $\R^n$ to $\R^n$, such that:
\[
\leftidx{^\tau}{T}{^\sigma_i}(x) = 
r_i^{\sigma(i) \tau(i,\sigma(i))} + P_i^{\sigma(i) \tau(i,\sigma(i))} x  \enspace. 
\]
We have 
\[ T(u+\alpha \nu) = \leftidx{^{\tau^*}}{\!T}{^{\sigma^*}}(u+\alpha \nu)
\enspace .
\]
The dynamic programming principle, for one-player games, implies that 
\[
\inf_{\sigma} J_i^k(\sigma,\tau^*) = (\leftidx{^{\tau^*}}{\!T}{})^k_i(0)
\enspace ,
\]
where the infimum is taken over all strategies of player \MIN.
By~\eqref{e-invhalf}, $\alpha \mapsto u+ \alpha \nu$ is an invariant
half-line of $\leftidx{^{\tau^*}}{\!T}{}$, and so
\[
\lim_k \inf_{\sigma}\frac{1}{k} J_i^k(\sigma,\tau^*) = 
\lim_k 
\frac{1}{k}(\leftidx{^{\tau^*}}{\!T}{})^k_i(0)
= \chi_i (\leftidx{^{\tau^*}}{\!T}{})= \nu_i \enspace,
\]
showing that the second equality in~\eqref{e-def-optuniform}
holds, with $v^{\text{U}}=\nu$.
Considering $T^{\sigma^*}$ instead
of $\leftidx{^{\tau^*}}{\!T}{}$ and arguing by duality, we deduce 
that the first equality in~\eqref{e-def-optuniform}
also holds, so that $\sigma^*,\tau^*$
are uniform optimal policies, which implies that the uniform value exists.
\end{proof}

Let us mention that the mean-payoff vector and the uniform value exist,
more generally, when $T$ is semialgebraic~\cite{BK76} or even definable in an o-minimal structure~\cite{BGV15}.
However, the existence of $\chi(T)$ is not guaranteed in general: a recent result of Vigeral~\cite{Vig13} shows that the limit may not exist even with a compact action space and transition payments and probabilities that are continuous with respect to the actions. 

Finally, observe  that the existence of an invariant
half-line $\alpha \mapsto u+ \alpha \nu$ where $\nu= \lambda e$ is a constant
vector is equivalent to the solvability of the ergodic equation $T(u)= u+ \lambda e$.
In the present setting, this implies that the solvability of the ergodic equation is equivalent
to the mean-payoff vector being independent of the initial state, that is
$\chi_i(T) = \chi_j(T)$ for every $i, j \in \state$.
Moreover, in this special case, the optimal policies $\sigma^*,\tau^*$ constructed
in Theorem~\ref{th-inv} are obtained by selecting minimizing
and maximizing actions in the expression of $T(x)$ in~\eqref{eq:Shapley-operator}, when
$x$ is replaced by $u$.

\section{Generic uniqueness of the bias vector of stochastic games}
\label{sec:generic-uniqueness}

\subsection{Statement of the main result}
\label{sec:main-result}

Let us first recall some definitions.
A {\em polyhedron} in $\R^n$ is an intersection of finitely many 
closed half-spaces,
a {\em face} of a polyhedron is an intersection of this polyhedron
with a supporting half-space, and a {\em polyhedral complex} is a finite set $\mathcal{K}$ of polyhedra satisfying the two following properties:
\begin{enumerate}
  \item $P \in \mathcal{K}$ and $F$ is a face of $P$ implies that $F \in \mathcal{K}$;
  \item for all $P, Q \in \mathcal{K}$, $P \cap Q$ is a face of $P$ and $Q$.
\end{enumerate}
A polyhedron in $\mathcal{K}$ is called a {\em cell} of the polyhedral complex.
We refer to the textbook~\cite{DLRS10} for background on polyhedral complexes.

Also, a map over $\R^n$ is said to be {\em piecewise affine} if $\R^n$ can be covered by a finite union of polyhedra (with nonempty interior) on which its restriction is affine.
in that case, the set of such polyhedra can be refined in a polyhedral complex.
In~\cite{Ovc02,AT07}, it is shown that the piecewise affine functions are exactly the ones that are defined as in~\eqref{eq:Shapley-operator}, i.e., the functions that can be written as a minimax over finite sets of affine functions.

Finally, we introduce the following definition, that extends the notion of (finite) ergodic Markov chain.
\begin{definition}\label{def-ergodicity}
  A stochastic game with finite state space and Shapley operator $T: \R^n \to \R^n$ is {\em ergodic} if the ergodic equation~\eqref{eq:ergodic-equation} is solvable for all operators $g+T$ with $g \in \R^n$.
\end{definition}
The authors have given in~\cite{AGH15} necessary and sufficient conditions for a perfect-information stochastic game with finite state space and bounded payment function to be ergodic, conditions that we partly recall in the next subsection.
Note however that in the case of finite games, the existence of an invariant half-line for piecewise affine maps readily implies that the latter definition is equivalent to the fact that the mean-payoff vector is constant for all additive perturbation vectors of the transition payments.

We now state the main result of this paper, the proof of which is postponed to Subsection~\ref{sec:proof-main-thm}.
\begin{theorem}
  \label{thm:generic-uniqueness}
  Let $T: \R^n \to \R^n$ be the Shapley operator of a finite stochastic game with perfect information.
  Assume that the game is ergodic.
  Then, the space $\R^n$ can be covered by a polyhedral complex such that, for any additive perturbation vector $g \in \R^n$ in the interior of a full-dimensional cell, $g+T$  has a unique bias vector, up to an additive constant.

  In particular, the set of perturbation vectors $g$ for which $g+T$ has more than
  one bias vector, up to an additive constant, is included in a
  finite union of subspaces of codimension at least $1$. 
\end{theorem}

\begin{remark}
  This perturbation theorem bears some conceptual similarity
  with results of weak KAM theory; we refer to the monograph by Fathi~\cite{Fat}
  for more information. The latter theory deals
  with a class of one-player deterministic games with continuous
  time and space. 
  In this setting, the bias vector $u$ and the eigenvalue $\lambda$ are solution
  of an ergodic Hamilton-Jacobi PDE $H(x,D_x u) = \lambda$ where the Hamiltonian
  $(x,p)\mapsto H(x,p)$ is convex in the adjoint variable $p$. 
  One may consider
  the perturbation of a Hamiltonian by a potential,
  which amounts to replacing $H(x,p)$ by $H(x,p)+V(x)$, for
  some function $V$. This is similar to the replacement
  of the Shapley operator $T$ by the perturbed Shapley operator $g+T$
  in Theorem~\ref{thm:generic-uniqueness}.
  As observed by Figalli and Rifford in~\cite[Th.~4.2]{rifford},
  it follows from weak KAM theory results that under
  some assumptions, the solution $u$ of $V(x)+H(u,D_x u)=\lambda$ is unique
  up to an additive constant for a generic function $V$. 
  Theorem~\ref{thm:generic-uniqueness} shows that an analogous
  property is valid for finite two-player zero-sum stochastic
  games. We note however that Theorem~\ref{thm:generic-uniqueness}
  does not extend easily to the case of PDE, since zero-sum 
  games correspond to Hamilton-Jacobi PDE with a {\em nonconvex} Hamiltonian
  (to which current weak KAM methods do not apply). 
\end{remark}

\subsection{Nonlinear spectral theory}
\label{sec:nonlinear-spectral-theory}

The purpose of this subsection is to present or extend some known results in nonlinear spectral theory that will be useful to prove Theorem~\ref{thm:generic-uniqueness}, as well as further results in Section~\ref{sec:applications}.

\subsubsection{Recession operator and ergodicity conditions}

To characterize the ergodicity of a perfect-information stochastic game with finite state space, we shall use, along the lines of~\cite{GG04,AGH15}, 
the {\em recession operator} associated with the Shapley operator $T: \R^n \to \R^n$. 
This operator is a self-map of $\R^n$ defined by
\begin{equation}
  \label{eq:recession-operator}
  \widehat{T}(x) := \lim_{\alpha \to +\infty} \frac{T(\alpha x)}{\alpha} \enspace , \quad x \in \R^n \enspace.
\end{equation}

Its existence is not guaranteed in general, but it does exist when the game is finite, i.e., when $T$ is piecewise affine (and more generally when the payment function is bounded).
In this case, $\widehat{T}$ is the Shapley operator of a modified version of the stochastic
game represented by $T$ where the transition payments are set to $0$.
Indeed, if $T$ is given as in~\eqref{eq:Shapley-operator}, then it is readily seen that the $i$th coordinate map of $\widehat{T}$ is
\[
  \widehat{T}_i (x) = \min_{a \in A_i} \max_{b \in B_{i,a}} P_i^{a b} x \enspace , \quad x \in \R^n \enspace.
\]
In this case, we also easily get that
\begin{equation*}
  \label{eq:recession-operator-uniform-convergence}
  \| T - \widehat{T} \|_\infty \leq \|r\|_\infty \enspace ,
\end{equation*}
which implies in particular that the limit~\eqref{eq:recession-operator} defining $\widehat{T}$ is uniform in $x$.

Observe that if $\widehat{T}$ exists, then it inherits from $T$ the additive homogeneity and the monotonicity properties.
Furthermore, it is positively homogeneous, meaning that $\widehat{T}(\alpha x) = \alpha \widehat{T}(x)$ for every $\alpha \geq 0$.
As a consequence, any vector proportional to the unit vector of $\R^n$ is a fixed point of $\widehat{T}$.
We shall call such fixed points {\em trivial} fixed points.
The following result relates the ergodicity of a game to the fixed points of the recession operator associated with its Shapley operator. 
Since it applies to games with bounded payment function, it deals a fortiori with the case of finite games.
\begin{theorem}[{\cite[Th.~3.1]{AGH15}}]
  Let $\Gamma$ be a perfect-information stochastic game with finite state space and bounded payment function.
  Let $T: \R^n \to \R^n$ be the Shapley operator.
  The following are equivalent:
  \begin{enumerate}
    \item the recession operator has only trivial fixed points;
    \item the mean-payoff vector exists and is constant for all additive perturbation vectors of the transition payments;
    \item the ergodic equation~\eqref{eq:ergodic-equation} has a solution for all Shapley operators $g+T$, $g \in \R^n$.
  \end{enumerate}
  \label{thm:ergodicity-condition}
\end{theorem}

\subsubsection{Characterization of the ergodic constant}
\label{sec:ergodic-constant-characterization}

Let $\Gamma$ be a finite stochastic game with perfect information,
and let $T$ be its Shapley operator.
We shall make use of
the reduced Shapley operator $T^\sigma: \R^n \to \R^n$ associated with the policy $\sigma \in \polMIN$ of player \MIN, defined in~\eqref{e-def-Tsigma}.
We also have, for all $x \in \R^n$,
\begin{equation}
  \label{eq:stochastic-control-operator}
  T^\sigma(x) = \max_{\tau \in \polMAX} \big( r^{\sigma \tau} + P^{\sigma \tau} x \big) \enspace ,
\end{equation}
where $r^{\sigma \tau}$ is the vector in $\R^n$ whose $i$th entry is defined by
$r^{\sigma \tau}_i = r^{\sigma(i) \tau(i,\sigma(i))}_i$ and $P^{\sigma \tau}$ is
the $n \times n$ stochastic matrix whose $i$th row is given by
$P^{\sigma \tau}_i = P^{\sigma(i) \tau(i,\sigma(i))}_i$.
Observe that $T^\sigma$ is convex (componentwise), monotone and additively homogeneous.

If $P$ is a $n \times n$ stochastic matrix, the {\em directed graph} associated with $P$ is composed of the nodes $1,\dots,n$ and of the arcs $(i,j)$, $1 \leq i,j \leq n$, such that $P_{i j} > 0$.
A {\em class} of the matrix $P$ is a maximal set of nodes such that every two nodes in the set are connected by a directed path.
A class is said to be {\em final} if every path starting from a node of this class remains in it.
Let us denote by $\M(P)$ the set of invariant probability measures of $P$, i.e., the set of stochastic (column) vectors $m \in \R^n$ such that $\transpose{m} \, P = \transpose{m}$.
Given a final class $C$ of $P$, there is a unique invariant probability measure $m \in \M(P)$
the support of which is $C$, i.e., $\{1 \leq  i \leq n \mid m_i > 0 \} = C$.
Moreover, the set $\M(P)$ is the convex hull of such measures.
Since the number of final classes of $P$ is finite, $\M(P)$ is a convex polytope.

Let us denote by $\chibar(T)$ the {\em upper mean payoff} of $T$, i.e., the greatest entry
of the mean-payoff vector $\chi(T)$.
We next give a characterization of $\chibar(T)$.
Obviously, if $T$ satisfies the ergodic equation~\eqref{eq:ergodic-equation},
then $\chi(T)$ is a constant vector and the eigenvalue is $\lambda(T) = \chibar(T)$.
In the sequel, we denote by $\<x,y>$ the standard scalar product in $\R^n$
of two vectors $x$ and $y$.

\begin{lemma}
  \label{lem:chibar-characterization}
  Let $T: \R^n \to \R^n$ be the Shapley operator of a finite stochastic game with perfect information $\Gamma$.
  Then the upper mean payoff of $T$ is given by
  \begin{equation}
    \label{eq:chibar}
    \chibar(T) = \min_{\sigma \in \polMIN} \max \{ \< m , r^{\sigma \tau} > \mid \tau \in \polMAX, \; m \in \M(P^{\sigma \tau}) \} \enspace .
  \end{equation}
\end{lemma}

\begin{proof}
  First, observe that for all policies $\sigma \in \polMIN$, we have $T \leq T^\sigma$,
  which yields, by monotonicity of the operators, $\chi(T) \leq \chi(T^\sigma)$,
  and in particular $\chibar(T) \leq \chibar(T^\sigma)$.

  Considering an invariant half-line~\eqref{eq:invariant-half-line} of $T$, we know that
  there exist a vector $u \in \R^n$ such that $T(u) = u + \chi(T)$.
  Let $\sigma \in \polMIN$ be a policy of player \MIN such that $T(u) = T^\sigma(u)$.
  Then, we have $T^\sigma(u) \leq u + \chibar(T) \unit$.
  Furthermore, we know by a Collatz-Wielandt formula (see~\cite[Prop.~1]{GG04}) that,
  for any monotone and additively homogeneous map $F: \R^n \to \R^n$, we have
  \[
    \chibar(F) = \inf \{ \mu \in \R \mid \exists x \in \R^n, \; F(x) \leq \mu \unit + x \} \enspace .
  \]
  So, we deduce that $\chibar(T^\sigma) \leq \chibar(T)$, and finally that
  \[
    \chibar(T) = \min_{\sigma \in \polMIN} \chibar(T^\sigma) \enspace .
  \]

  Now we fix a policy $\sigma$ of player \MIN, and we let $\chi := \chi(T^\sigma)$ and
  \[
    \mu^{\sigma} := \max \{ \< m , r^{\sigma \tau} > \mid
    \tau \in \polMAX, \; m \in \M(P^{\sigma \tau}) \} \enspace .
  \]
  Since $T^\sigma$ has an invariant half-line with direction $\chi$, there is a vector $v \in \R^n$
  such that $T^\sigma(v + \alpha \chi) = v + (\alpha+1) \chi$ for all $\alpha \geq 0$.
  In particular, for every policy $\tau \in \polMAX$ we have
  \[
    r^{\sigma \tau} + P^{\sigma \tau} v \leq T^\sigma(v) =
    v + \chi \leq v + \chibar(T^\sigma) \unit \enspace .
  \]
Multiplying this inequality by any $m \in \M(P^{\sigma \tau})$,
 we deduce that $\mu^\sigma \leq \chibar(T^\sigma)$.

Furthermore, since the germs of affine functions from $\R$ to $\R$ at infinity
  are totally ordered, there exists a policy $\tau \in \polMAX$ such that
  \[
    T^\sigma(v + \alpha \chi) = r^{\sigma \tau} + P^{\sigma \tau} (v + \alpha \chi)
  \]
  for all $\alpha$ large enough.
  In particular, since the equality
  \begin{equation}
    \label{eq:halfline-linear-map}
    v + (\alpha+1) \chi = r^{\sigma \tau} + P^{\sigma \tau} (v + \alpha \chi)
  \end{equation}
  holds for all $\alpha$ large enough, we get that $P^{\sigma \tau} \chi = \chi$.
  Thus, $\chi$ is an {\em harmonic vector} for the stochastic matrix $P^{\sigma \tau}$,
  and as such it is constant on any final class of $P^{\sigma \tau}$,
  and its maximum is attained on one of these final class (see~\cite[Lem.~2.9]{AG03}).
  Let $m \in \M(P^{\sigma \tau})$ be the invariant probability measure associated
  with a final class $C$ of $P^{\sigma \tau}$ such that $\chi_i = \chibar(T^\sigma)$ for all $i \in C$.
  Then, we deduce from~\eqref{eq:halfline-linear-map} that $\<m,r^{\sigma \tau}> = \chibar(T^\sigma)$,
  which yields $\mu^\alpha \geq \chibar(T^\sigma)$ and finally $\mu^\alpha = \chibar(T^\sigma)$.
\end{proof}

Let us mention that in~\eqref{eq:chibar}, the set of invariant probability measures of
$P^{\sigma \tau}$, $\M(P^{\sigma \tau})$, may be replaced by the set of its extreme points,
denoted by $\M^*(P^{\sigma \tau})$, since it is a convex polytope.
Note that $\M^*(P^{\sigma \tau})$ is the set of invariant probability measures, the support
of which are the final classes of $P^{\sigma \tau}$.

\subsubsection{Structure of the eigenspace}

An ingredient of our approach is a result of \cite{AG03} which describes the eigenspace
of {\em one-player} Shapley operators $T$.
We next recall this result.
We assume that $T$ arises from a game in which only player \MAX has nontrivial actions,
so that player \MIN has only one possible policy.
We also assume that $T$ satisfies the ergodic equation~\eqref{eq:ergodic-equation}
so that its eigenvalue $\lambda(T)$ is equal to the entries of the mean-payoff vector
$\chi(T)$ which is constant.
In this case, the representation of the upper-mean payoff $\chibar(T)$, hence of
the eigenvalue $\lambda(T)$, in Lemma~\ref{lem:chibar-characterization} simplifies
as the dependency in $\sigma$ can be dropped, and we arrive, with a trivial simplification
of the notation, to
\begin{align}
  \lambda(T) = \max \{ \< m , r^{\tau} > \mid \tau \in \polMAX, \; m \in \M(P^{ \tau}) \} \enspace .
  \label{eq:eigenvalue-convex}
\end{align}

It is shown in~\cite{AG03} that the dimension of the eigenspace of $T$ is controlled
by the number of {\em critical classes}.
The latter can be defined through the notion of maximizing measures, which rely on
{\em randomized} policies.
For every state $i$, such a policy $\tau$ assigns to every action $b \in B_i$ a probability
$\tau(i,b)$ that this action is selected.
This leads to the stochastic matrix $P^\tau$ with entries
$P^\tau_{i j} = \sum_{b \in B_i} P^{b}_{i j} \, \tau(i,b)$, and to the payment vector
$r^{\tau}$ with entries $r^\tau_i = \sum_{b \in B_i} r^{b}_{i} \, \tau(i,b)$.
Then, the maximum in~\eqref{eq:eigenvalue-convex} is unchanged if it is taken over the set
of randomized policies $\tau$ and of invariant probability measures $m$ of the corresponding
stochastic matrix $P^\tau$ (see~\cite[Prop.~7.2]{AG03}).
A measure $m$ is {\em maximizing} if $\lambda(T) = \<m,r^\tau>$ for some randomized policy
$\tau$, and if $m \in \M^*(P^\tau)$, that is, the support of $m$ is a final class of $P^\tau$.
A subset $I\subset \{1,\dots,n\}$ is a {\em critical class} if there exists
a maximizing measure $m$ whose support is $I$, i.e., $I = \{ 1 \leq i \leq n \mid m_i>0 \}$,
and if $I$ is a maximal element with respect to inclusion among all the subsets of
$\{1,\dots,n\}$ which arise in this way.
Note that critical classes are disjoint.

The following lemma gives a sufficient condition for the critical class to be unique,
Note that the result does not require randomized stationary strategies.
\begin{lemma}
  Let $T: \R^n \to \R^n$ be a convex, monotone and additively homogeneous map.
  Suppose that the ergodic equation~\eqref{eq:ergodic-equation} is solvable.
  If there is a unique probability measure which attains the maximum
  in~\eqref{eq:eigenvalue-convex}, then $T$ has a unique critical class.
  \label{lem:uniqueness-critical-class}
\end{lemma}

\begin{proof}
  Let $C$ be a critical class of $T$.
  There exists a randomized policy $\tau$ such that $C$ is a final class of $P^\tau$ and
  the unique invariant probability measure $m$ with support $C$ satisfies
  $\lambda = \lambda(T) = \<m,r^\tau>$.

  Let $u$ be an eigenvector.
  For every $i \in \{1,\dots,n\}$, we have $\lambda + u_i - r^\tau_i - P^\tau_i u \geq 0$.
  Furthermore, since $\transpose{m} \, P^\tau = \transpose{m}$, we also have
  \[
    \sum_{i \in C} m_i \, ( \lambda + u_i - r^\tau_i - P^\tau_i u) = 
    \lambda + \<m,u> - \<m,r^\tau> - \<m,(P^\tau u)>  = 0 \enspace .
  \]
  This yields that $\lambda + u_i - r^\tau_i - P^\tau_i u = 0$ for every $i \in C$.

  Let $\tau'$ be a deterministic policy such that $\tau'(i) = b_i$ with $\tau(i,b_i) > 0$
  for all indices $i$.
  Since $P^{\tau}$ can be written as a convex combination
  with positive coefficients of the $P^{\tau'}$ with such $\tau'$,
  we deduce that $C$ contains a final class $C'$ of $P^{\tau'}$, and
  so there exists an invariant probability measure $m'$ of $P^{\tau'}$ whose support $C'$ 
  is included in $C$.
  Similarly $r^{\tau}$ is a convex combination of the $r^{\tau'}$
  with same coefficients as for $P^{\tau}$. Hence,
  we deduce by a similar argument as above that
  $\lambda + u_i - r^{\tau'}_i - P^{\tau'}_i u = 0$
  for every $i \in C$.
  Multiplying by $m'$, we obtain that $\lambda = \<m',r^{\tau'}>$,
  which implies that $m'$ attains the maximum in~\eqref{eq:eigenvalue-convex}.
  Since such a probability measure is unique, $m'$ and $C'$ are the same for all
  $\tau'$ as above, hence $C=C'$ is the support of the unique
  probability measure attaining the maximum in~\eqref{eq:eigenvalue-convex}.
\end{proof}

We now describe the eigenspace of $T$.
\begin{theorem}[{\cite[Th.~1.1]{AG03}}]
  \label{thm:convex-spectral-theorem}
  Let $T: \R^n \to \R^n$ be a convex, monotone and additively homogeneous map.
  Suppose that the ergodic equation~\eqref{eq:ergodic-equation} is solvable.
  and let $C\subset \{1,\dots,n\}$ be the set of elements of critical classes.
  We denote by $\pi_C$ the restriction map $\R^n\to \R^C$,
  $x\mapsto (x_i)_{i\in C}$.
  Then, the set of eigenvectors of $T$, denoted by $\E(T)$, satisfies the following properties:
  \begin{enumerate}
    \item Every element $x$ of $\E(T)$ is uniquely determined by its restriction $\pi_C(x)$.
    \item The set $\pi_C(\E(T))$ is convex and its dimension is at most equal to the number
      of critical classes of $T$; moreover, the latter bound is attained when
      $T$ is piecewise affine. 
  \end{enumerate}
\end{theorem}

In particular, combined with Lemma~\ref{lem:uniqueness-critical-class}, the above result
yields the following.
\begin{corollary}
  Let $T: \R^n \to \R^n$ be a convex, monotone and additively homogeneous map.
  Suppose that the ergodic equation~\eqref{eq:ergodic-equation} is solvable.
  If there is a unique probability measure which attains the maximum
  in~\eqref{eq:eigenvalue-convex}, then $T$ has a unique eigenvector up to an additive constant.
  \label{coro:uniqueness-eigenvector}
\end{corollary}

We refer the reader to~\cite{AG03} for more background on critical classes, which admit
several characterizations and can be computed in polynomial time when the game is finite.
We only provide here a simple illustration in order to understand the latter theorem.

\begin{example}
  Let $T: \R^2\to\R^2$ be such that 
  \[
    T_1(x) = \max \Big\{ x_1, \frac{x_1+x_2}{2} \Big\} \enspace ,
    \quad T_2(x) = \max \Big\{ -3+x_2, \frac{x_1+x_2}{2} \Big\} \enspace .
  \]
  We have $T(0)=0$, which shows in particular that the upper mean payoff is $0$.
  If player \MAX chooses, when in state $1$,
  the action corresponding to the first term
  in the expression of $T_1$, and when in state $2$, the action
  corresponding to the second term
  in the expression of $T_2$, we arrive at the transition
  matrix 
  \[
    P= \begin{pmatrix} 1 & 0 \\ 1/2 & 1/2\end{pmatrix}
  \]
  which has the invariant probability measure $m=\transpose{(1,0)}$.
  This measure attains the maximum in~\eqref{eq:eigenvalue-convex}.
  However, its support $I=\{1\}$ is not a critical class for it is not
  maximal with respect to inclusion.
  Indeed, if player \MAX chooses instead, when in state $1$, the action corresponding
  to the second term in the expression of $T_1$, we arrive at the transition matrix
  \[
    P = \begin{pmatrix} 1/2 & 1/2 \\ 1/2 & 1/2\end{pmatrix}
  \]
  which has the invariant probability measure $m=\transpose{(1/2,1/2)}$.
  This measure attains the maximum in~\eqref{eq:eigenvalue-convex} and its support
  $I=\{1,2\}$ is maximal with respect to inclusion.
  Hence, $I=\{1,2\}$ is the unique critical class.
  It follows from  Theorem~\ref{thm:convex-spectral-theorem} that $0$ is the unique
  eigenvector of $T$, up to an additive constant.
\end{example}

Another ingredient is a variant of a result of Bruck~\cite{Bru73},
concerning the topology of fixed-point sets of nonexpansive maps.
We now consider a {\em two-player} Shapley operator $T$ such that the ergodic
equation~\eqref{eq:ergodic-equation} is solvable, and denote by
\[
  \E(T):= \{u \in \R^n \mid T(u)= \lambda e + u\}
\]
the set of eigenvectors of $T$ (recall that the eigenvalue $\lambda$ is unique). 
\begin{theorem}[{Compare with~\cite[Th.~2]{Bru73}}]
  \label{thm:nonexpansive-retract}
  Let $T: \R^n \to \R^n$ be a monotone and additively homogeneous map.
  Assume that the ergodic equation~\eqref{eq:ergodic-equation} 
  is solvable.
  Then, the set of eigenvectors $\E(T)$ is a retract of $\R^n$ by a 
  sup-norm nonexpansive map, meaning that $\E(T)=p(\R^n)$ where $p$
  is a sup-norm nonexpansive self-map of $\R^n$ such that $p=p^2$.
  In particular, $\E(T)$ is arcwise connected.
\end{theorem}

\begin{proof}
  The result of Bruck~\cite[Th.~2]{Bru73} shows that, under some compactness conditions,
  the fixed-point set of a nonexpansive self-map of a Banach space is a retract
  of the whole space by a nonexpansive map. 

  Assume now that $T$ is monotone, additively homogeneous, and admits an eigenvector
  for the eigenvalue $\lambda$.
  Then, the eigenspace $\E(T)$ coincides with the fixed-point set
  of the map $x\mapsto -\lambda \unit  +T(x)$.
  The latter map is sup-norm nonexpansive and satisfies the condition of~\cite[Th.~2]{Bru73},
  and so, $\E(T)$ is a nonexpansive retract of $\R^n$. 
\end{proof}

\begin{remark}
  The retraction $p$ in Theorem~\ref{thm:nonexpansive-retract}
  can be chosen to be monotone and additively homogeneous.
  This can actually be shown by elementary means, following a construction
  in the proof of~\cite[Lem.~3]{GG04}.
  Indeed, we may assume without loss of generality that $\lambda=0$, and consider
  $q(x):= \lim_{k \to \infty} \inf_{\ell \geq k} T^\ell(x)$, which is finite
  because every orbit of a nonexpansive map that admits a fixed point must be bounded.
  Since $T$ is monotone and continuous, we get $T(q(x)) \leq q(x)$, and so,
  $p(x):= \lim_{k\to \infty} T^k(q(x))$, which is the limit of a nonincreasing and
  bounded sequence, exists and is finite.
  The map   $p$ is easily shown to be monotone and additively homogeneous and to satisfy $p=p^2$. 
\end{remark}

\subsection{Proof of Theorem~\ref{thm:generic-uniqueness}}
\label{sec:proof-main-thm}

Let $T$ be the Shapley operator of a finite stochastic game with perfect information $\Gamma$ which is assumed to be ergodic.
Let $\sigma \in \polMIN$ be a policy of player \MIN.
We define the real map $\lambda^\sigma(\cdot)$ on $\R^n$ by
\begin{equation}
  \label{eq:eigenvalue-one-player}
  \lambda^\sigma(g) := \max \{ \< m, (g+r^{\sigma \tau}) > \mid
  \tau \in \polMAX, \; m \in \M^*(P^{\sigma \tau}) \} \enspace ,
\end{equation}
where $\M^*(P)$ denotes the set of extreme points of the convex polytope $\M(P)$,
that is, the set of invariant probability measures associated
with the final classes of the stochastic matrix $P$.
The fact that $\M^*(P^{\sigma \tau})$ is a set of probability measures yields that $\lambda^{\sigma}$ is monotone and additively homogeneous, hence sup-norm nonexpansive (and continuous).
Furthermore, since the set of policies $\polMAX$ of player \MAX is finite, as well as
all the sets $\M^*(P^{\sigma \tau})$, then the map $\lambda^\sigma$ is piecewise affine.

We now define the polyhedral complex $\C^\sigma$ covering $\R^n$, the full-dimensional cells of
which are precisely the maximal polyhedra on which the piecewise affine map
$\lambda^\sigma$ coincides with an affine map.

Let $Q$ be a cell of $\C^\sigma$ with full dimension.
We claim that if a vector $g$ is in the interior of $Q$, then the set of eigenvectors of the reduced one-player Shapley operator $F := g+T^\sigma$ is either empty or reduced to a line.
To see this, it suffices to observe that the measure $m$ attaining
the maximum in~\eqref{eq:eigenvalue-one-player}
is unique for all $g$ in the interior of $Q$ and independent of the choice of 
$g$ in this interior. Indeed, if $m$ is such a measure, 
then  $m\in\M^*(P^{\sigma \tau})$ for some $\tau \in \polMAX$
and for $d=\< m, r^{\sigma \tau} > \in\R$, we have 
$\lambda^\sigma(g')\geq\<m,g'> + d$,
for all $g'\in \R^n$, with equality at $g$.
Since $g'\mapsto \lambda^\sigma(g')$ is an affine map on $Q$,
and $g$ is in the interior of $Q$, the equality  
$\lambda^\sigma(g')=\<m,g'> + d$ holds for all $g'\in Q$ and so
$m$ must coincide with the linear part of the affine map,
which is independent of $g$ (and unique).
We deduce from Corollary~\ref{coro:uniqueness-eigenvector} that $\E(F)$, if it is nonempty,
is reduced to a line of direction $\unit$.

Consider now the polyhedral complex $\C$ obtained as the refinement of all the complexes $\C^\sigma$.
This complex still covers $\R^n$ and has cells with nonempty interior.
Let $g$ be a perturbation vector in the interior of a full-dimensional cell of $\C$.
Since the game $\Gamma$ is ergodic, $\E(g+T)$ is not empty.
Let $u$ be an eigenvector of $g+T$.
According to~\eqref{eq:T-min}, there is a policy $\sigma \in \polMIN$ of player \MIN such that $g+T(u) = g+T^\sigma(u)$.
Hence $u$ is also an eigenvector of $g+T^\sigma$.
So, there is a finite family $\Sigma^*$ of $\polMIN$ such that $\E(g+T) = \bigcup_{\sigma \in \Sigma^*} \E(g+T^\sigma)$.
Moreover, we have proved that for any policy $\sigma \in \Sigma^*$, the eigenspace $\E(g+T^\sigma)$ is reduced to a line.
Thus, $\E(g+T)$ is composed of a finite union of lines which all have the
same direction, $\unit$. Consider now the hyperplane orthogonal to 
the unit vector, $H:=\{x\in \R^n\mid \<x,\unit> =0\}$, and let $\pi$ denote the orthogonal projection on $H$.
Then, $\pi(\E(g+T))=\E(g+T)\cap H$ is finite.
However, by Theorem~\ref{thm:nonexpansive-retract}, 
$\E(g+T)$ is connected. Then, the set $\pi(\E(g+T))$ is also connected,
and since it is finite, it must be reduced to a point. 
It follows that $g+T$ has 
a unique eigenvector, up to an additive constant. 
\qed

\subsection{Example}

We conclude this section by an example illustrating Theorem~\ref{thm:generic-uniqueness}.
Consider the following Shapley operator defined on $\R^3$ (here we use $\wedge$ and $\vee$ instead of $\min$ and $\max$, respectively, and we recall that the addition has precedence over them):
\begin{equation*}
  T(x) =
  \begin{pmatrix}
    \frac{1}{2} (x_1 + x_3) \, \wedge \, 1 + \frac{1}{2} (x_1 + x_2)\\
    2 + \frac{1}{2} (x_1 + x_3) \, \wedge \, \left(1 + \frac{1}{2} (x_1 + x_2) \, \vee \, -2 + x_3 \right)\\
    3 + \frac{1}{2} (x_1 + x_3) \, \vee \, 1 + x_3
  \end{pmatrix} \enspace .
\end{equation*}
It can be proved, using Theorem~\ref{thm:ergodicity-condition}, that the ergodic
equation~\eqref{eq:ergodic-equation} is solvable for every perturbation vector $g \in \R^3$.
Figure~\ref{fig} shows the intersection of the hyperplane $\{g \in \R^3 \mid g_3 = 0\}$ and
the polyhedral complex introduced in Theorem~\ref{thm:generic-uniqueness}.
Here, for each vector $g$ in the interior of a full-dimensional polyhedron, $g+T$ has a unique eigenvector up to an additive constant.

\begin{figure}[!h]
  \begin{center}
    \begin{tikzpicture}[scale=0.15]
      % coordinate axes
      \draw [->, >= angle 60,black] (-5,0) -- (20,0);0
      \draw (19,0) node[above] {$g_1$};
      \foreach \x in {0,5,10,15}
      \draw [black] (\x,0) -- (\x,-0.5);
      \draw (10,0) node[below] {\small $10$};
      \draw [->, >= angle 60,black] (0,-20) -- (0,7);
      \draw (0,6) node[right] {$g_2$};
      \foreach \y in {-15,-10,-5,0,5}
      \draw [black] (-0.5,\y) -- (0,\y);
      \draw (0,-10) node[left] {\small $-10$};
      \draw (15,6) node {$g_3=0$};
      % complex 1: (r_1+r_6)/2 = -3/2+(g_1+g_3)/2   |   r_7 = 1+g_3
      \draw [thick,blue] (5,-20) -- (5,7);
      % complex 2: (2*r_2+r_3+r_6)/4 = (1+2*g_1+g_2+g_3)/4   |   r_7 = 1+g_3
      \draw [thick,blue] (-2,7) -- (23/2,-20);
      % complex 3: (r_2+r_4)/2 = 1+(g_1+g_2)/2   |   r_7 = 1+g_3   |   (2*r_2+r_5+2*r_6)/5 = (-6+2*g_1+g_2+2*g_3)/5
      \draw [thick,blue] (11,-11) -- (-5,5);
      \draw [thick,blue] (11,-11) -- (20,-42/3);
      \draw [thick,blue] (11,-11) -- (16,-20);
      \draw (0,0) node {$\bullet$};
      \draw (0,0) node [below left] {\small $0$};
    \end{tikzpicture}
  \end{center}
  \caption{}
  \label{fig}
\end{figure}

Let us detail what happens in the neighborhood of $g=0$, point in which $g+T$ fails to have a unique eigenvector.
Note that in the neighborhood of $g=0$, the eigenvalue of $g+T$ remains $1$.
\begin{itemize}
  \item If $g_1+g_2 = 0$, the eigenvectors of $g+T$ are defined by
    \[ x_1=x_2+2 g_1 \enspace , \quad -3+g_2 \leq x_2-x_3 \leq -2-g_1 \enspace . \]
  \item If $g_1+g_2 > 0$, the unique eigenvector, up to an additive constant, is
    \[ (-2+2 g_1,-2+2 g_1+2 g_2,0) \enspace . \]
  \item If $g_1+g_2 < 0$, the unique eigenvector, up to an additive constant, is
    \[ (-3+2 g_1+g_2,-3+g_2,0) \enspace . \]
\end{itemize}

\section{Application to policy iteration}
\label{sec:applications}

We finally apply our results to show that policy iteration combined
with a perturbation scheme can solve degenerate stochastic games.

\subsection{Hoffman-Karp policy iteration}

Let us recall the notation of Sections~\ref{sec:preliminaries}
and~\ref{sec:generic-uniqueness}.
We denote by $\Gamma$ a finite stochastic game with perfect information.
The (finite) set of policies of player \MIN
is denoted by $\polMIN$, i.e., an element
of $\polMIN$ is a map $\sigma: \state \to \bigcup_{i \in \state} A_i$ such that
$\sigma(i) \in A_i$ for every state $i \in \state$.
The (finite) set of policies of player \MAX
is denoted by $\polMAX$, i.e., an element of $\polMAX$ is a map 
$\tau: K_A \to \bigcup_{i \in \state, a \in A_i} B_{i,a}$
such that $\tau(i,a) \in B_{i,a}$ for every state $i \in \state$ and action $a \in A_i$.
Finally, recall that for $\sigma \in \polMIN$ and $\tau \in \polMAX$,
$P^{\sigma \tau}$ denotes the $n \times n$ stochastic matrix whose $i$th row is given by
$P^{\sigma \tau}_i = P^{\sigma(i) \tau(i,\sigma(i))}_i$.

When $T:\R^n \to \R^n$ is the Shapley operator~\eqref{eq:Shapley-operator} of
a finite stochastic game with perfect information, 
Hoffman and Karp~\cite{HK66} have introduced
a policy iteration algorithm,
which takes the description of the game as the input
and returns the eigenvalue $\lambda$ and
an eigenvector $u$ of $T$, i.e., a solution $(\lambda,u) \in \R \times \R^n$ of
the ergodic equation $T(u) = \lambda \unit + u$.
Also, optimal stationary strategies for both players in the mean-payoff game can be 
derived from the output of the algorithm.

This algorithm and more generally policy iteration procedures are a standard general way
for solving mean-payoff stochastic games.
In worst-case scenario, they require an exponential number of iterations,
see~\cite{friedmann}.
However, no polynomial-time algorithm is known to solve mean-payoff games.
In fact, it is an important open question to know whether there exists one, since
this problem, or any problem involving a polynomial-time equivalent perfect-information
two-player zero-sum stochastic game (such as discounted stochastic games,
simple stochastic games, parity games, see~\cite{AM09}) is one of the few problems that
belong to the complexity class NP~$\cap$~coNP, see~\cite{Con92}.

It is convenient here to state an abstract, slightly more general, version
of the Hoffman-Karp algorithm, described in terms of the
operators $T$ and $T^\sigma$ (Algorithm~\ref{algo:main}). 

\begin{algorithm}
  \SetKwInOut{Input}{input}\SetKwInOut{Output}{output}
  \DontPrintSemicolon
  \Input{Shapley operator $T$ of perfect-information finite stochastic game.
  }
  \Output{eigenvalue $\lambda$ and eigenvector $u$ of $T$.}
  \BlankLine
  \KwSty{initialization}: select an arbitrary policy $\sigma_0 \in \polMIN$\;
  \Repeat{$\sigma_{\ki+1} = \sigma_{\ki}$}{
    compute an eigenpair $(\lambda^\ki,v^\ki)$ of $T^{\sigma_\ki}$\; \label{it:value-search}
    improve the policy $\sigma_\ki$ in a conservative way:
    select a policy $\sigma_{\ki+1} \in \polMIN$ such that
    $T^{\sigma_{\ki+1}}(v^\ki) = T(v^\ki)$ with, for every state $i \in \state$
    such that $T_i^{\sigma_\ki}(v^\ki) = T_i(v^\ki)$,
    $\sigma_{\ki+1}(i) = \sigma_\ki(i)$\;
  }
  \KwSty{return} $\lambda^\ki$ and $v^\ki$\;
  \caption{Policy iteration, compare with~\cite{HK66}}
  \label{algo:main}
\end{algorithm}

We assume Algorithm~\ref{algo:main} is interpreted in exact
arithmetics (the vectors $v^\ki$ have rational coordinates
and the $\lambda^\ki$ are rational numbers). To implement
Step~\ref{it:value-search}, we may call any oracle able
to compute the eigenvalue and an eigenvector of a one-player stochastic game. 
In the original approach of Hoffman and Karp, the oracle
consists in applying the same policy iteration algorithm
for the one-player game with fixed policy $\sigma_\ki$.
The proof of Hoffman and Karp shows that Algorithm~\ref{algo:main}
is valid under a restrictive assumption.

\begin{theorem}[{Corollary of~\cite{HK66}}]\label{th-HK}
  Algorithm~\ref{algo:main} terminates and is correct if for all choices
  of policies $\sigma$ and $\tau$ of the two players, the corresponding transition
  matrix $P^{\sigma\tau}$ is irreducible. 
\end{theorem}

Indeed, it is easy to see that
the sequence $(\lambda^\ki)_k$ of Algorithm~\ref{algo:main} is
nonincreasing, that is, $\lambda^{\ki+1} \leq \lambda^\ki$ for all iterations $\ki$.
The irreducibility assumption was shown to imply that the latter inequalities are always strict,
which entails the finite time convergence (each policy
yields a unique well defined eigenvalue, these eigenvalues
constitute a decreasing sequence, and there are finitely
many policies).

However, the assumption that {\em all} 
the stochastic matrices $P^{\sigma \tau}$ be irreducible
is way too strong to guarantee that Algorithm~\ref{algo:main}
is properly posed. Indeed, to execute the algorithm,
it suffices that at every iteration $\ki$
the operator $T^{\sigma_\ki}$ admits an eigenvalue and an eigenvector,
which is the case in particular if for all policies
$\sigma$, the graph obtained by taking the {\em union} of the edge sets
of all the graphs associated to $P^{\sigma \tau}$ for the different choices of $\tau$ is strongly connected
(the existence of the eigenvalue and eigenvector,
in this generality, goes back to Bather~\cite{bather}, see also
\cite{GG04,AGH15} for a more general discussion). 
In particular, the irreducibility assumption
of Hoffman and Karp is essentially never satisfied for deterministic
games, whereas the condition involving the union of the edge
sets is satisfied by relevant classes of deterministic games.

It should be noted that Algorithm~\ref{algo:main} may, in general,
lead to degenerate iterations, in which $\lambda^{\ki+1}=\lambda^\ki$.
As shown by an example in~\cite[Sec.~6]{ACTDG}, this may lead the algorithm to cycle
when the bias vector is not unique.
This difficulty was solved first in the deterministic framework in~\cite{CTGG99},
where it was shown that cycling can be avoided by enforcing a special choice of the bias vector,
obtained by a nonlinear projection operation.
This approach was then extended to the stochastic framework in~\cite{CTG06,ACTDG}. 
As a special case of these results, we get that policy iteration 
is correct and does terminate under much milder conditions than in Theorem~\ref{th-HK}.

\begin{theorem}[Corollary of~{\cite[Th.~7]{CTG06}}]
  \label{thm:PI-termination}
  Algorithm~\ref{algo:main} terminates and is correct if for each choice
  of policy $\sigma$ of player \MIN, the operator $T^{\sigma}$ has an eigenvalue and a {\em unique} eigenvector, up to an additive constant.
\end{theorem}

We next show that the conditions of Theorem~\ref{thm:PI-termination} are satisfied
for generic payments,
and conclude that nongeneric instances can still be solved
by the Hoffman-Karp algorithm, after an effective perturbation of the input. 

\subsection{Generic termination of policy iteration}

Let $T:\R^n \to \R^n$ be the Shapley operator of
a finite stochastic game with perfect information.
The following assumption guarantees that Algorithm~\ref{algo:main} is well posed
for any additive perturbation of $T$.
\begin{assumption}
  \label{asm:ergodicity}
  For any policy $\sigma \in \polMIN$ of player \MIN, the one-player game
  with Shapley operator $T^{\sigma}$ is ergodic
  in the sense of Definition~\ref{def-ergodicity},
  meaning that for all perturbation
  vectors $g \in \R^n$, the operator $g+T^{\sigma}$ has an eigenvalue (and an eigenvector).
\end{assumption}
This assumption is much milder than the original assumption
of Hoffman and Karp, requiring all the transition matrices
$P^{\sigma \tau}$ to be irreducible (see Theorem~\ref{th-HK}).
Also, since $\widehat{T} = \min_{\sigma \in \polMIN} \widehat{T^\sigma}$, it readily follows
from Theorem~\ref{thm:ergodicity-condition} that Assumption~\ref{asm:ergodicity}
implies that the original two-player game is ergodic.
Moreover, we shall see in the next subsection that one can always transform
a game (in polynomial time) by a ``big $M$'' trick in such a way that
Assumption~\ref{asm:ergodicity} becomes satisfied.

By using the arguments of Section~\ref{sec:generic-uniqueness}, we now show that under
Assumption~\ref{asm:ergodicity}, Algorithm~\ref{algo:main}
terminates for a generic perturbation of the payments.

\begin{theorem}
  Let $T: \R^n \to \R^n$ be the Shapley operator of a finite stochastic game
  with perfect information satisfying Assumption~\ref{asm:ergodicity}.
  Then, the space $\R^n$ can be covered by a polyhedral complex such that
  for each additive perturbation vector $g \in \R^n$ in the interior of
  a full-dimensional cell, Algorithm~\ref{algo:main} terminates after a finite number of
  steps and gives an eigenpair of $g+T$.
  \label{thm:generic-termination}
\end{theorem}

\begin{proof}
  Consider the same complex $\C$ as in Section~\ref{sec:generic-uniqueness} and
  let $g$ be a perturbation vector in the interior of a full-dimensional cell of $\C$.
  It follows from the proof of Theorem~\ref{thm:generic-uniqueness}
  (see Subsection~\ref{sec:proof-main-thm}) that for any policy $\sigma \in \polMIN$,
  the eigenvector of $g+T^{\sigma}$, which exists according to
  Assumption~\ref{asm:ergodicity}, is unique up to an additive constant.
  Hence, at each step $\ki$ of Algorithm~\ref{algo:main}, the bias vector $v^\ki$
  of $g+T^{\sigma_{\ki}}$ is unique up to an additive constant.
  The conclusion follows from Theorem~\ref{thm:PI-termination}.
\end{proof}

We next provide an explicit perturbation $g$, depending on a parameter $\varepsilon$,
for which the policy iteration algorithm applied to $g+T$ is valid.
We shall see in the next subsection that $\varepsilon$ can be instantiated with a polynomial
number of bits, in such a way that the original unperturbed problem is solved.

Before giving this explicit perturbation scheme, let us mention that the polyhedral complex
$\C$ introduced in the proof of Theorem~\ref{thm:generic-uniqueness}
(Subsection~\ref{sec:proof-main-thm}) can be constructed for any perfect-information
finite stochastic game, whether it is ergodic or not.
Recall indeed that $\C$ is obtained as a refinement of all the regions where
the maps $g \mapsto \lambda^\sigma(g)$ with $\sigma \in \polMIN$, defined
in~\eqref{eq:eigenvalue-one-player}, are affine.
In particular, we do not assume ergodicity in the next two propositions.

\begin{proposition}
  \label{prop:perturbation-scheme}
  Let $T: \R^n \to \R^n$ be the Shapley operator of a finite stochastic game
  with perfect information.
  Then, there exists $\varepsilon_0 > 0$ such that all perturbation vectors
  $g_\varepsilon := (\varepsilon, \varepsilon^2, \dots, \varepsilon^n)$
  with $0 < \varepsilon < \varepsilon_0$ are in the interior of
  the same full-dimensional cell of the polyhedral complex $\C$ introduced
  in Subsection~\ref{sec:proof-main-thm}.
\end{proposition}

\begin{proof}
  The cells of the polyhedral complex $\C$ introduced in
  Subsection~\ref{sec:proof-main-thm}
  that are not full-dimensional are included in an arrangement of a finite number of hyperplanes.
  The real curve $\varepsilon \mapsto g_\varepsilon= (\varepsilon,\dots,\varepsilon^n)$
  cannot cross a given hyperplane in this arrangement more than $n$ times
  (otherwise, a polynomial of degree $n$ would have strictly more than $n$ roots).
  We deduce that there is a value $\varepsilon_0 > 0$ such that the restriction
  of the curve $\varepsilon \mapsto g_\varepsilon$ to the open interval $(0,\varepsilon_0)$
  crosses no hyperplane of the arrangement.
  Therefore, it must stay in the interior of a full-dimensional cell of the complex $\C$.
\end{proof}

The following result is a refinement of the previous one. 

\begin{proposition}
  \label{prop:optimal-policies}
  Let $T: \R^n \to \R^n$ be the Shapley operator of a finite stochastic game
  with perfect information.
  Then, there exist $\varepsilon_1 > 0$, policies $\sigma \in \polMIN$ and
  $\tau \in \polMAX$, and an invariant probability measure $m^{\sigma \tau}$ of
  the stochastic matrix $P^{\sigma \tau}$ such that for all $\varepsilon \in [0,\varepsilon_1]$,
  the upper mean payoff of $g_\varepsilon + T$ is given by
  $\chibar(g_\varepsilon + T) = \<m^{\sigma \tau}, (g_\varepsilon + r^{\sigma \tau})>$.
\end{proposition}

\begin{proof}
  Recall that the upper mean payoff of $g+T$ is given by
  \begin{gather*}
    \label{e-min}
    \chibar(g+T) = \min_{\sigma \in \polMIN} \lambda^\sigma(g) \\
    \label{e-max}
    \text{with} \qquad
    \lambda^\sigma(g) = \max \{ \< m, (g+r^{\sigma \tau}) > \mid
    \tau \in \polMAX, \; m \in \M^*(P^{\sigma \tau}) \} \enspace .
  \end{gather*}
  By construction of the polyhedral complex $\C$ in
  Subsection~\ref{sec:proof-main-thm},
  in the interior of a full-dimensional cell, each piecewise affine map
  $g \mapsto \lambda^\sigma(g)$ coincides with a unique affine map, but
  $g \mapsto \chibar(g+T)$ need not be affine.
  Hence, we can refine the complex $\C$ into a complex $\C'$ such that the latter piecewise
  affine map also coincides with a unique affine map on each cell of full dimension.
  The same proof as Proposition~\ref{prop:perturbation-scheme} leads to the existence
  of a parameter $\varepsilon_1$ such that all the perturbations $g_\varepsilon$
  with $\varepsilon \in (0,\varepsilon_1)$ lie in the interior of the same full-dimensional
  cell of $\C'$.
  Let $Q$ be this cell.
  By construction of the latter complex, there exists a policy $\sigma \in \polMIN$
  such that $\chibar(g+T) = \lambda^\sigma(g)$ for all $g \in Q$, and for that policy $\sigma$
  there exists a policy $\tau \in \polMAX$ and $m^{\sigma \tau} \in \M^*(P^{\sigma \tau})$
  such that $\lambda^\sigma(g) = \<m^{\sigma \tau},(g+r^{\sigma \tau})>$ for all $g \in Q$.
\end{proof}

It follows that solving the game with Shapley operator $g_\varepsilon +T$
for $\varepsilon$ small enough entails a solution of the original game.

\begin{proposition}
  If $T$ satisfies Assumption~\ref{asm:ergodicity}, then there exists $\varepsilon_1 > 0$
  (same as in Proposition~\ref{prop:optimal-policies}) such that Algorithm~\ref{algo:main}
  terminates for any input $g_\varepsilon + T$ with $\varepsilon \in (0,\varepsilon_1)$.
  Furthermore, for any policy $\sigma$ satisfying $\lambda(g+T) = \lambda(g+T^\sigma)$,
  we also have $\lambda(T) = \lambda(T^\sigma)$.
\end{proposition}

\begin{proof}
  First note that here, since the game is ergodic, we have $\chibar(g+T) = \lambda(g+T)$
  for all perturbations $g \in \R^n$.

  Following Proposition~\ref{prop:optimal-policies}, the parameter $\varepsilon_1$ is
  such that all perturbations $g_\varepsilon$ with $0 < \varepsilon < \varepsilon_1$
  lie in the interior of the same full-dimensional cell of the complex $\C'$
  (which is a refinement of $\C$, see the proof of Proposition~\ref{prop:optimal-policies}).
  Denote by $Q$ this cell.
  The termination of Algorithm~\ref{algo:main} is then a straightforward consequence of
  Theorem~\ref{thm:generic-termination}.

  By definition of the polyhedral complex $\C'$, the piecewise affine map
  $g \mapsto \lambda(g+T)$ is affine when restricted to $Q$.
  Furthermore, for any policy $\sigma \in \polMIN$, we have either
  $\lambda(g+T) = \lambda^\sigma(g) = \lambda(g+T^\sigma)$ for all $g$ in $Q$,
  or $\lambda(g+T) < \lambda^\sigma(g) = \lambda(g+T^\sigma)$ for all $g$ in
  the interior of $Q$.
  Hence the result.
\end{proof}

\subsection{Complexity issues}

In this subsection, we show that computing the upper mean payoff of a Shapley operator
(a fortiori the eigenvalue if it exists) is polynomial-time reducible to the computation
of the eigenvalue of a Shapley operator for which Algorithm~\ref{algo:main} terminates.
This fact is a direct consequence of Theorem~\ref{thm:polynomial-reduction} below. To do so, we shall need explicit bounds on the perturbation
parameter $\varepsilon$.

We first explain how the general case can be reduced to the situation
in which Assumption~\ref{asm:ergodicity} holds.
To that purpose, let use introduce for any real number $M \geq 0$, the map
$R_M: \R^n \to \R^n$ whose $i$th coordinate is given by
\[
  [R_M (x)]_i = \max \Big\{ x_i, \max_{1 \leq j \leq n} (-M + x_j)  \Big\} \enspace ,
  \quad x \in \R^n \enspace .
\]
It is convenient to introduce {\em Hilbert's seminorm} on $\R^n$, defined by
\[
  \Hnorm{x} = \max_{1 \leq i \leq n} x_i - \min_{1 \leq i \leq n} x_i \enspace .
\]
Observe that $R_M$ is a projection on the set $\{ x \in \R^n \mid \Hnorm{x} \leq M \}$,
meaning that $R^2_M = R_M$ and
\[
  R_M(x) = x \iff \Hnorm{x} \leq M \enspace .
\]

\begin{lemma}
  \label{lem:game-reduction}
  Let $T:\R^n \to \R^n$ be the Shapley operator of a perfect-information
  finite stochastic game, and let $M \geq 0$.
  Then, $T \circ R_M$ has an eigenvalue.
\end{lemma}

\begin{proof}
  First, note that the recession operator of $\widehat{R}_M$ is given by
  \[
    \widehat{R}_M (x) = (\max x) \, \unit \enspace , \quad x \in \R^n \enspace ,
  \]
  where $\max x := \max_{1 \leq i \leq n} x_i$.
  Second, it has been noted in Subsection~\ref{sec:nonlinear-spectral-theory} that
  the limit~\eqref{eq:recession-operator} defining $\widehat{T}$ is uniform in $x$.
  Hence, we get that $\widehat{T \circ R}_M = \widehat{T} \circ \widehat{R}_M$.
  Thus, using the properties of recession operators, we have, for any vector $x \in \R^n$,
  \[
    \widehat{T \circ R}_M (x) = \widehat{T} \circ \widehat{R}_M (x)
    = \widehat{T} \big( (\max x) \, \unit \big) = (\max x) \, \unit \enspace .
  \]
  This proves that the only fixed points of $\widehat{T \circ R}_M$
  are trivial fixed points.
  The conclusion follows from Theorem~\ref{thm:ergodicity-condition}.
\end{proof}

Given a perfect-information finite stochastic game $\Gamma$ with Shapley operator $T$,
the operator $T \circ R_M$ can be interpreted as the Shapley operator of another
perfect-information finite stochastic game with state space $\state$.
In this game, at each step, if the current state is $i \in \state$, player \MIN start
by choosing an action $a \in A_i$.
Then, player \MAX chooses an action $b \in B_{i,a}$ which gives rise to
a transition payment $r_i^{a b}$ and a state, $j$, is chosen by nature with probability
$[P_i^{a b}]_j$ and announced to the players.
Finally, player \MAX chooses the next state to be either $j$ with no additional payment,
or to be any other state $k$ with an additional payment of $-M$.
In other words, player \MAX has the option of teleporting himself to any other state,
by accepting a penalty $M$. 

Note that, since player \MIN has the same action space in the latter game
as in the game $\Gamma$, the sets of her stationary strategies in both games are identical.
Then, for a fixed policy $\sigma$ of player \MIN, the one-player Shapley operator
$(T \circ R_M)^\sigma$ is equal to $T^\sigma \circ R_M$,
and we get from Lemma~\ref{lem:game-reduction} the following result.

\begin{corollary}
  \label{coro:ergodicity-assumption}
  Let $T:\R^n \to \R^n$ be the Shapley operator of
  a perfect-information finite stochastic game.
  Then, $T \circ R_M$ satisfies Assumption~\ref{asm:ergodicity}.
\end{corollary}

In the modified game with Shapley operator $T \circ R_M$, player \MAX makes, at each step,
the final decision about the next state, provided an additional cost of $M$.
The following result shows that if this cost is large enough,
then player \MAX cannot do better, in the long run, than in the game $\Gamma$.

\begin{lemma}
  \label{lem:homogeneization}
  Let $T:\R^n \to \R^n$ be the Shapley operator of
  a perfect-information finite stochastic game.
  Then, there exists a positive constant $M_0$ such that for any $M > M_0$,
  the eigenvalue of $T \circ R_M$ is equal to the upper mean payoff $\chibar(T)$.
\end{lemma}

\begin{proof}
  First, note that $R_M(x) \geq x$ for all $x \in \R^n$.
  Hence, by monotonicity of $T$, we deduce that $T \circ R_M \geq T$,
  which yields $\chi(T \circ R_M) \geq \chi(T)$.
  Since $T \circ R_M$ has an eigenvalue, denoted by $\lambda(T \circ R_M)$,
  then we have $\lambda(T \circ R_M) \geq \chibar(T)$.

  Second, we know that $T$ has an invariant half-line with direction $\chi(T)$.
  So there exists a vector $u \in \R^n$ such that $T(u) = u + \chi(T)$.
  Now let $M_0 := \Hnorm{u}$.
  For every $M > M_0$, we have
  \[
    T \circ R_M (u) = T(u) = u + \chi(T) \leq u + \chibar(T) \unit \enspace .
  \]
  By application of a Collatz-Wielandt formula (see~\cite{GG04}), we know that the eigenvalue
  of $T \circ R_M$ is given by
  \[
    \lambda(T \circ R_M) = \inf \{ \mu \in \R \mid
    \exists u \in \R^n, \; T \circ R_M(u) \leq u + \mu \unit \} \enspace .
  \]
  Hence $\lambda(T \circ R_M) \leq \chibar(T)$.
\end{proof}

We shall need a technical bound on invariant probability measures of stochastic matrices
arising from strategies.
We state it here for an arbitrary irreducible stochastic matrix.

\begin{lemma}
  \label{lem-matrix}
  Let $P$ be a $n\times n$ irreducible stochastic matrix whose entries are rational
  numbers with numerators and denominators bounded by an integer $D$.
  Then, the entries of the invariant probability measure of $P$ are rational
  numbers whose least common denominator is bounded by
  \begin{equation*}
    \label{eq:bound-denominators}
    n^{n/2} D^{n^2} 
  \end{equation*}
\end{lemma}

\begin{proof}
  The invariant probability measure $m$ of $P$ is the unique solution of the linear system
  \begin{equation}
    \label{eq:invariant-measure}
    \begin{cases}
      (I - \transpose{P}) \, m = 0\\
      \transpose{\unit} \, m = 1
    \end{cases} \enspace ,
  \end{equation}
  where $I$ is the identity matrix.
  Note that one row of the subsystem $(I - \transpose{P}) \, m = 0$ is redundant
  since we are dealing with stochastic vectors and $\transpose{\unit} \, m = 1$.
  Then, by deleting this row, and by multiplying every row of the latter subsystem
  by all the denominators of the coefficients appearing in this row, we arrive at
  a Cramer linear system with integer coefficients of absolute value less than $ D^n$,
  and with unit coefficients on the last row. 
  Solving this system by Cramer's rule, we obtain that the entries of $m$ are rational
  numbers whose denominators divide the determinant of the system.
  Using Hadamard's inequality for determinants, we deduce that these denominators are bounded by
  \begin{flalign*}
    && big( (n-1) (D^n)^2 + 1 \big)^{n/2}
    \leq n^{n/2} D^{n^2} 
    \enspace . && \qed
  \end{flalign*}
\renewcommand{\qedsymbol}{}
\end{proof} 

Let $T: \R^n \to \R^n$ be the Shapley operator of a perfect-information finite stochastic game
$\Gamma$ with state space $\{1,\dots,n\}$.
We just showed that the upper mean payoff of $T$ can be recovered
from the eigenvalue of the operator $T \circ R_M$ 
for a suitable large $M$.
The latter operator satisfies Assumption~\ref{asm:ergodicity}, and so, we can in principle apply Algorithm~\ref{algo:main} to it. 
However, to do so in a way which leads to a polynomial-time transformation
of the input, it is convenient to introduce the following
modified Shapley operator $T_M: \R^{2n} \to \R^
{2n}$
given, for all $(x,y) \in \R^n \times \R^n$, by
\[
  T_M(x,y) := \left( T(y),R_M(x) \right) \enspace . 
\]
Note that we have
\begin{equation*}
  (T_M)^2(x,y) =
  \begin{pmatrix}
    T \circ R_M(x) \\ R_M \circ T (y)
  \end{pmatrix} \enspace .
  \label{eq:T^2_M}
\end{equation*}
The following immediate lemma shows that one can recover the bias
vectors and the eigenvalue of $T\circ R_M$ from those of $T_M$.
\begin{lemma}\label{lemma-transform}
If $v$ is a bias vector of the operator $T\circ R_M$
with eigenvalue $\lambda$, then $(v,R_M(v)-(\lambda/2) \unit )$ is a bias
vector of the operator $T_M$ with eigenvalue $\lambda/2$,
and all bias vectors of $T_M$ arise in this way.\hfill\qed
\end{lemma}

The operator $T_M$ is the dynamic programming operator of a game, denoted by $\Gamma_M$,
with state space $\{1,\dots,2n\}$.
In each state $i \in \{1,\dots,n\}$, the actions, the payments and the transition function
are the same as in $\Gamma$, except that the next state is labeled
by an element of $\{n+1,\dots,2n\}$ instead of $\{1,\ldots,n\}$. 
Moreover, in each state $i \in \{n+1,\dots,2n\}$, player \MIN has only one possible action, while
player \MAX chooses the next state $j$ among $\{1,\dots,n\}$ with a cost $M$ if $i-j \neq n$.
In particular, the policies of player \MIN 
in the two games $\Gamma$ and $\Gamma_M$ are in one-to-one correspondence, and
to simplify the presentation, we shall use the same notation for these policies.
Hence, we shall write $(T_M)^\sigma(x,y) = (T^\sigma(y),R_M(x))$ for such a policy $\sigma$.

We saw in Lemma~\ref{lem:game-reduction} that the operator
$T\circ R_M$ has an eigenvalue.
The same is true for the operator $T_M$ by Lemma~\ref{lemma-transform}.
Moreover, the same conclusion applies to the operator $(T_M)^\sigma$ for any
policy $\sigma$, so that $T_M$ satisfies Assumption~\ref{asm:ergodicity}.
We know that for $\varepsilon>0$ small enough, the perturbed
operator $g_\varepsilon +T\circ R_M$ has a unique bias vector,
up to an additive constant,
where $g_\varepsilon = (\varepsilon,\dots,\varepsilon^n)$.
This leads to considering, for $\varepsilon > 0$, 
\[
  T_{M,\varepsilon} := (g_\varepsilon,0) + T_M \enspace .
\]

\begin{theorem}
  \label{thm:polynomial-reduction}
  Let $\Gamma$ be a perfect-information finite stochastic game whose transition payments
  and probabilities are rational numbers with numerators and denominators bounded
  by an integer $D\geq 2$.
  Let $T:\R^n \to \R^n$ be the Shapley operator of $\Gamma$.
  If
  \[
    M >  4 n^{n/2} D^{n^2+1} \quad 
\text{and} \quad
    0 < \varepsilon <  \frac{1}{ n^n D^{2 n (n+1)} } \enspace ,
  \]
  then the upper mean payoff of $T$ can be recovered from
  $T_{M,\varepsilon} = (g_\varepsilon,0) + T_M$, in the sense that for any policy $\sigma$
  of player \MIN such that $\lambda(T_{M,\varepsilon}) = \lambda((T_{M,\varepsilon})^\sigma)$,
  we have $\chibar(T) = \chibar(T^\sigma)$.
  Furthermore, such
  a policy can be obtained by applying Algorithm~\ref{algo:main}
  with the input $T_{M,\varepsilon}$.
\end{theorem}

\begin{proof}
  Let $g \in [0,1]^n$, and fix a policy $\sigma$ of player \MIN.
  In the game $\Gamma_M$, consider a policy of player \MAX such that,
  when in state $i \in \{1,\dots,n\}$, he chooses action $b_i \in B_{i,\sigma(i)}$,
  and when in state $i \in \{n+1,\dots,2n\}$, he chooses
  the next state to be $j(i) \in \{1,\dots,n\}$.
  Then, the transition matrix associated with that choice of policy is
  the following $2n \times 2n$ block matrix:
  \begin{equation}
    \label{eq:transition-marix}
    \begin{pmatrix}
      0 & P^{\sigma \tau} \\ Q & 0
    \end{pmatrix} \enspace ,
  \end{equation}
  where $\tau \in \polMAX$ is a policy of player \MAX in the game $\Gamma$
  such that $\tau(i,\sigma(i)) = b_i$ for each state $i$, and where $Q$ is the $n \times n$
  stochastic matrix whose coefficients are $Q_{i j} = 1$ for $j=j(i+n)$ and $0$ otherwise.

  Let $(m,m') \in \R^n \times \R^n$ be an invariant probability measure of the
  stochastic matrix~\eqref{eq:transition-marix}.
  The vectors $m$ and $m'$ satisfy in particular
  \begin{equation}
    \label{eq:invariant-measures}
    \transpose{m}\,  P^{\sigma \tau} =  \transpose{m'}\enspace , \quad  \transpose{m'} Q =  \transpose{m} \enspace ,
    \quad \<m,e> + \<m',e> = 1 \enspace .
  \end{equation}
  We are interested in the eigenvalue of the perturbed one-player Shapley operator
  $(g,0) + (T_M)^\sigma$.
  Hence, following formula~\eqref{eq:eigenvalue-convex}, we consider the quantity
  \[
    \gamma := \<m,(g+r^{\sigma \tau})> -
    M \sum_{\substack{1 \leq i \leq n\\ Q_{i i} = 0}} m'_i \enspace .
  \]

  If for every index $i$ in the support of $m'$, we have $Q_{i i} = 1$,
  then we deduce from the second equality in~\eqref{eq:invariant-measures} that $m' = m$.
  This yields that $2m$ is an invariant probability measure of $P^{\sigma \tau}$, and that
  \[
    \gamma = \<m,(g+r^{\sigma \tau})> \leq
    \frac{1}{2} \, \chibar(g+T^\sigma) \enspace .
  \]
  Note that the equality is attained in the above inequality
  for some policy $\tau$ and some invariant probability measure $m$.

  If there is an index $i$ in the support of $m'$ such that $Q_{i i} = 0$, then we have
  \[
    \gamma \leq \<m,(g+r^{\sigma \tau})> - M \, m'_i 
    \leq 1 + \max_{i,a,b} r_i^{a b} - K_1 \, M \enspace ,
  \]
  where $K_1$ is a positive constant such that $K_1 \leq m'_i$.
  Note that $K_1$ can be chosen independently of $M$ and of the particular choices of
  the policies.
  Then, taking $M > M_0 := (K_1)^{-1} (1 + (3/2) \|r\|_\infty)$, we obtain that
  \[
    \gamma < \frac{1}{2} \, \min_{i,a,b} r_i^{a b}
    \leq \frac{1}{2} \, \chibar(g+T^\sigma) \enspace .
  \]

  Thus, we have proved that, for all policies $\sigma$ of player \MIN
  and for all $g \in [0,1]^n$, we have
  \[
    \lambda \big( (g,0)+(T_M)^\sigma \big) =
    \frac{1}{2} \, \chibar(g+T^\sigma) \enspace ,
  \]
  as soon as $M > M_0$.
  In particular, the choice of the parameter $\varepsilon$ such that $T_{M,\varepsilon}$
  is a generic instance only relies on $T$ (Proposition~\ref{prop:optimal-policies}).

  We now fix some $M > M_0$.
  Consider, for policies $\sigma, \sigma'$ of player \MIN and $\tau, \tau'$ of player \MAX,
  two distinct pairs $(m,d) \neq (m',d')$, where $m \in \M^*(P^{\sigma \tau})$,
  $m' \in \M^*(P^{\sigma' \tau'})$, $d := \<m, r^{\sigma \tau}>$
  and $d' := \<m', r^{\sigma' \tau'}>$.
  We need to compare the affine maps $g \mapsto \<m,g>+d$ and $g \mapsto \<m',g>+d'$
  along the curve $\varepsilon \mapsto g_\varepsilon$ with $\varepsilon \in (0,1)$.

  Assume first that $d = d'$.
  Then $m \neq m'$ and we can select the smallest index $i$ such that $m_i \neq m'_i$.
  Note that since $m$ and $m'$ are stochastic vectors, we necessarily have $i < n$
  and we also have the existence of another index $j$ such that $i < j \leq n$ and $m_j \neq m'_j$.
  Without loss of generality, we may assume that $m_i - m'_i > 0$.
  Let $K_2 \in \R$ be such that $0 < K_2 < m_i - m'_i$.
  Then, for any positive parameter $\varepsilon < K_2 \, n^{-1}$, we have
  \begin{multline*}
    (\<m,g_\varepsilon> + d) - ( \<m',g_\varepsilon> + d' )
    = (m_i - m'_i) \varepsilon^i + \sum_{i < j \leq n} (m_j-m'_j) \varepsilon^j \\
    > K_2 \, \varepsilon^i - n \varepsilon^{i+1} = \varepsilon^i \, (K_2 - n \varepsilon)
    > 0 \enspace .
  \end{multline*}
  Assume now that $d \neq d'$, say $d > d'$, and let $K_3 \in \R$ be such that $0 < K_3 < d - d'$.
  Then, for any positive parameter $\varepsilon < K_3$, we have
  \[
    (\<m,g_\varepsilon> + d) - ( \<m',g_\varepsilon> + d' )
    > \varepsilon^n - \varepsilon + K_3
    > 0 \enspace .
  \]
  Note that we can choose the positive constants $K_2$ and $K_3$ independently of
  $\sigma$, $\sigma'$, $\tau$, $\tau'$, $m$ and $m'$.
  Hence, the above arguments show that the set of polynomial functions
  $\varepsilon \mapsto \<m,(g_\varepsilon+r^{\sigma \tau})>$  with $\sigma \in \polMIN$,
  $\tau \in \polMAX$, and $m \in \M^*(P^{\sigma \tau})$, is totally ordered
  if $\varepsilon$ is restricted to the interval $(0 , \min \{K_2 n^{-1},K_3\} )$.
  Thus, the parameter $\varepsilon_1$ of Proposition~\ref{prop:optimal-policies}
  may be taken equal to $\min \{ K_2 n^{-1},K_3 \}$.

  To complete the proof, we next explain how to instantiate the constants $K_1$ to $K_3$.
  Let us start with $K_2$.
  It is a lower bound on the absolute values of the differences between two distinct
  entries (with same index) of invariant probability measures associated with
  the transition matrices of $\Gamma$.
  These differences are of the form $|p_1/q_1-p_2/q_2| \geq 1/(q_1q_2)$,
  where $p_1,p_2,q_1,q_2$ are integers.
  By Lemma~\ref{lem-matrix}, we know that $q_1,q_2 \leq n^{n/2}D^{n^2}$, and so
one can choose
  \[
    K_2 = \frac{1}{ n^n D^{2 n^2}} \enspace .
  \]

  Likewise, $K_3$ is a lower bound on the absolute values of the differences between
  two distinct scalar products $\<m,r^{\sigma \tau}>$.
  Let $m \in \M(P^{\sigma \tau})$.
  It follows from  Lemma~\ref{lem-matrix} that the $i$th entry of $m$ can be
  written as $m_i = p_i / q$ where $p_i$ is an integer and $q \leq n^{n/2} D^{n^2}$ is
  an integer independent of $i$.
  Since every entry of $r^{\sigma \tau}$ has a denominator at most $D$, it follows that
  $\<m,r^{\sigma \tau}>$ is a rational number with denominator at most $n^{n/2} D^{n^2} D^n$.
  Therefore, the difference between two distinct values of $\<m,r^{\sigma \tau}>$ is at least
  $K_3=(n^{n} D^{2n^2} D^{2n})^{-1}$, and so
  \[
    \min \{ K_2 n^{-1},K_3 \} =
    \frac{1}{n^n D^{2 n^2}} \min \left\{ \frac{1}{n}, \frac{1}{D^{2n}} \right\} 
    = \frac{1}{n^nD^{2n(n+1)}}\enspace .
  \]
  Finally the constant $K_1$ is a lower bound for the positive entries of 
the restrictions $m'$  of the invariant probability measures $(m,m')$
 of the transition matrices~\eqref{eq:transition-marix}
  arising in the game $\Gamma_M$.
  A direct application of Lemma~\ref{lem-matrix} provides the following coarse bound:
  \[
    K_1 = \frac{1}{(2n)^n D^{4 n^2}} \enspace .
  \]
  This bound can be improved by noting that the matrices~\eqref{eq:transition-marix} have a particular structure.
Indeed, if $(m,m')$  is an invariant probability measure
of~\eqref{eq:transition-marix}, then it satisfies~\eqref{eq:invariant-measures},
and so $2m'$ is an invariant probability measure of the matrix  $QP^{\sigma\tau}$,
the entries of which are entries of $P^{\sigma\tau}$ (since $Q$ has only 
one nonzero entry in each row and this entry is equal to $1$).
  Applying Lemma~\ref{lem-matrix} to the matrix $QP^{\sigma\tau}$,
we obtain the following lower bound for the positive entries of $m'$:
 \[
    K_1 = \frac{1}{2 n^{n/2} D^{n^2}} \enspace .
  \]
Hence, 
  \begin{flalign*}
    && M_0 \leq (1 + (3/2)D) \; 2 n^{n/2} D^{n^2}\leq 4 n^{n/2} D^{n^2+1} \enspace . && \qed
  \end{flalign*}
\renewcommand{\qedsymbol}{}
\end{proof}

An important special case to which the method of Theorem~\ref{thm:polynomial-reduction}
can be applied concerns {\em deterministic} mean-payoff games~\cite{gurvich,zwick}. 
The input of such games can be described, as in~\cite{AGGut10}, 
by means of two matrices $A, B \in (\Z \cup \{-\infty\})^{m \times n}$.
The corresponding Shapley operator can be written as
\begin{align}
  \label{e-det}
  T_i(x) = \min_{1 \leq j \leq m} \big( -A_{ji} + \max_{1 \leq k \leq n} (B_{jk} + x_k) \big)
  \enspace , \quad x \in \R^n \enspace , \quad 1 \leq i \leq n \enspace .
\end{align}
The corresponding game is played by moving a token on a graph in which $n$ nodes,
denoted by $1,\dots,n$, belong to player \MIN, whereas $m$ other nodes,
denoted by $1',\dots,m'$, belong to player \MAX.
In state $i \in \{1,\dots,n\}$, player \MIN can move the token to a state
$j \in \{1',\dots,m'\}$ such that $A_{ji}\neq -\infty$, receiving $A_{ji}$.
In state $j$, player \MAX can move the token to a state $k \in \{1,\dots,n\}$
such that $B_{jk} \neq -\infty$, receiving $B_{jk}$.
We assume that the matrix $B$ has no identically infinite row, and that the matrix $A$
has no identically infinite column, meaning that each player has at least
one available action in each state.
Then, the modified operator $T \circ R_M$ corresponds to the matrix $B_M$
in which infinite entries of $B$ are replaced by $-M$,
and the operator $g + T \circ R_M$ arises by subtracting the constant $g_i$
to every entry in the $i$th column of $A$.

\begin{theorem}
  \label{theo-det}
  Let $T$ denote the Shapley operator~\eqref{e-det} of a deterministic mean-payoff game,
  with integer payoffs bounded in absolute value by $D \geq 2$.
  Then, for
  \[
    M > 4 n D \quad \text{and} \quad 0 < \varepsilon < 1 / n^3 \enspace ,
  \]
  the policy iteration Algorithm~\ref{algo:main} applied to the operator
  $g_\varepsilon + T \circ R_M$ terminates,
  and we can compute the upper mean payoff of $T$ from any policy $\sigma$ of player \MIN
  such that $\lambda(g_\varepsilon + T \circ R_M) = \lambda(g_\varepsilon + T^\sigma \circ R_M)$,
  in the sense that $\chibar(T) = \chibar(T^\sigma)$.
\end{theorem}

\begin{proof}
  We adapt the proof of Theorem~\ref{thm:polynomial-reduction} to the case of
  deterministic transition matrices.
  In that special case, every invariant probability measure is uniform on its support, 
hence its positive entries are
  bounded below by $1/n$ if the state space has cardinality $n$.
  Since the constant $K_1$ is a lower bound for the positive entries of
  the invariant probability measures of the transition matrices in $\Gamma_M$,
one can choose $K_1 = 1/(2n)$, and then
  \[
    M_0 \leq (2n) (1 + (3/2) D) \leq 4 n D \enspace .
  \]
  The constant $K_2$, which is a lower bound on the absolute values of the differences
  between two distinct entries of invariant probability measures of transition matrices in $\Gamma$,
  can be chosen as $K_2 = 1 / n^2$.
  As for $K_3$, it is a lower bound on the absolute values of the differences
  between two distinct values of $\<m,r^{\sigma \tau}>$.
  Since the payments are integers, every scalar product $\<m,r^{\sigma \tau}>$ is
  a rational number whose denominator divides the denominator of the positive entries of $m$ which are themselves bounded by $n$.
  Hence, one can choose $K_3 =1 / n^2$, and the parameter $\varepsilon$ must be lower than
  \begin{flalign*}
    && \min \{ K_2 n^{-1}, K_3 \} = 1 / n^3 \enspace . && \qed
  \end{flalign*}
\renewcommand{\qedsymbol}{}
\end{proof}

\begin{remark}
  One step in Algorithm~\ref{algo:main} consists in computing an eigenpair
  $(\lambda^\ki,v^\ki)$ of the reduced Shapley operator $T^{\sigma_\ki}$
  obtained by fixing the strategy $\sigma_k$ of player \MIN.
  This is a simpler problem which can be solved by several known methods. 
  We may apply, for instance, a similar policy iteration algorithm to $T^{\sigma_\ki}$,
  iterating this time in the space of policies $\tau$ of player \MAX.
  In this way, for each choice of $\tau$, we arrive at an operator of the form
  $T^{\sigma_\ki,\tau} (x) = g + P x$, where $P$ is a stochastic matrix
  which cannot in general be assumed to be irreducible.
  However, for one-player problems, a classical version of policy iteration, the multichain
  policy iteration introduced by Howard~\cite{How60} and Denardo and Fox~\cite{DF68},
  does allow one to determine $(\lambda^\ki,v^\ki)$ (without genericity conditions). 
  Moreover, in the special case of deterministic games, the vector $v^\ki$ is known
  to be a tropical eigenvector and $\lambda^\ki$ is a tropical eigenvalue.
  The tropical eigenpair can be computed by direct combinatorial algorithms,
  see e.g.\ the discussion in~\cite{CTGG99}.
\end{remark}

\begin{remark}
  Theorem~\ref{theo-det} should be compared with the other known perturbation scheme,
  relying on vanishing discount.
  This method requires the computation of a fixed point of the operator
  $x \mapsto T(\alpha x)$ for $0 < \alpha < 1$ sufficiently close to one.
  It is known that, for deterministic mean-payoff games, if the discount factor $\alpha$
  is chosen so that
  \[
    \alpha > 1 - \frac{1}{4(n+m)^3 D} \enspace ,
  \]
  where $D$ denotes the maximal absolute value of a finite entry $A_{ij}$ or $B_{ij}$,
  then, the solution of the mean-payoff problem can be derived from the solution of
  the discounted problem, see~\cite[Sec.~5]{zwick}.
  The latter can be obtained by policy iteration (which terminates without any
  nondegeneracy conditions in the discounted case). 
  Applying Algorithm~\ref{algo:main} to the map $g_\varepsilon + T$ requires solving
  linear systems in which the matrix is independent of $\varepsilon$.
If this is done by calling the Denardo-Fox algorithm to solve one-player problem
(see previous remark), these systems are well conditioned.
By comparison, vanishing discount requires the inversion of a matrix which
  becomes singular as $\alpha\to 1$. 
  In particular, if policy iteration is interpreted in floating-point arithmetics,
  vanishing discount based perturbations may lead to numerical instabilities or to overflows,
  whereas the present additive perturbation scheme is insensitive to this pathology,
  because it only perturbs the right-hand sides of the linear systems to be solved.
\end{remark}

\begin{remark}
  The present approach allows one to compute the {\em upper mean payoff},
  i.e., the maximum of the mean payoff over all initial states. 
  This leads to no loss of expressivity since it follows
  from known reductions that this problem is polynomial time
  equivalent to solving a mean-payoff game in which the initial
  state is fixed: combine the results of Appendix C in the extended version of~\cite{AGS16}
  (especially Lemma C.2 and Corollary C.3) with the reductions in~\cite{AM09}.
  An alternative route to compute the mean payoff
  of a given initial state, avoiding the use of such reductions, would
  be to extend the present perturbation scheme to the ``multichain'' version
  of policy iteration, discussed in~\cite{CTG06,ACTDG}.
\end{remark}

\bibliographystyle{amsalpha}
\bibliography{references}

\end{document}